\pdfoutput=1
\RequirePackage{ifpdf}
\ifpdf 
\documentclass[pdftex]{sigma}
\else
\documentclass{sigma}
\fi

\usepackage{fancyhdr}
\usepackage{calc}
\usepackage{xparse}
\usepackage{tikz}
\usetikzlibrary{matrix,backgrounds}

\def\Z{\mathbb{Z}}
\def\S{\mathbb{S}}

\def\C{\mathbb{C}}

\def\V{\mathbb{V}}
\newcommand{\cali}{\mathcal{I}}
\newcommand{\calu}{\mathcal{U}}
\newcommand{\calh}{\mathcal{H}}

\DeclareMathOperator{\Hom}{\textup{Hom}}
\DeclareMathOperator{\Irr}{\textup{Irr}}

\newtheorem{Th}{Theorem}[section]
\newtheorem{Pro}[Th]{Proposition}
\newtheorem{Lem}[Th]{Lemma}
\newtheorem{Cor}[Th]{Corollary}
\theoremstyle{definition}
\newtheorem{Def}[Th]{Definition}
\newtheorem{Ex}[Th]{Example}
\newtheorem{Rem}[Th]{Remark}

\numberwithin{equation}{section}

\begin{document}

\newcommand{\arXivNumber}{2001.02164}

\renewcommand{\PaperNumber}{041}

\FirstPageHeading

\ShortArticleName{A Decomposition of Twisted Equivariant $K$-Theory}

\ArticleName{A Decomposition of Twisted Equivariant $\boldsymbol{K}$-Theory}

\Author{Jos\'e Manuel G\'OMEZ and Johana RAM\'IREZ}
\AuthorNameForHeading{J.M.~G\'omez and J.~Ram\'irez}
\Address{Escuela de Matem\'aticas, Universidad Nacional de Colombia, Medell\'in, Colombia}
\Email{\href{mailto:jmgomez0@unal.edu.co}{jmgomez0@unal.edu.co}, \href{mailto:jramirezg@unal.edu.co}{jramirezg@unal.edu.co}}

\ArticleDates{Received July 13, 2020, in final form April 15, 2021; Published online April 21, 2021}

\Abstract{For $G$ a finite group, a normalized 2-cocycle $\alpha\in Z^{2}\big(G,{\mathbb S}^{1}\big)$ and $X$ a $G$-space on which a normal subgroup $A$ acts trivially, we show that the $\alpha$-twisted $G$-equivariant $K$-theory of $X$ decomposes as a direct sum of twisted equivariant $K$-theories of $X$ parametrized by the orbits of an action of $G$ on the set of irreducible $\alpha$-projective representations of $A$. This generalizes the decomposition obtained in~[G\'omez~J.M., Uribe~B., \textit{Internat.~J. Math.} \textbf{28} (2017), 1750016, 23~pages, arXiv:1604.01656] for equivariant $K$-theory. We also explore some examples of this decomposition for the particular case of the dihedral groups $D_{2n}$ with~$n\ge 2$ an even integer.}


\Classification{19L50; 19L47}

\section{Introduction}

In the past few years there has been a growing interest in studying
twisted $K$-theory motivated by its appearance in string theory and also due
to the celebrated theorem of Freed, Hokpins and~Teleman (see~\cite[Theorem~1]{FHT}).
In this article we study $G$-equivariant twisted $K$-theory when~$G$ is a finite group.
Our main goal is to show that, under suitable hypothesis, the canonical decomposition
theorem for projective representations can be used to obtain a decomposition for~twisted
$G$-equivariant $K$-theory as a direct sum of other twisted equivariant $K$-theories,
thus generalizing the work in~\cite{GomezUribe}, where a similar decomposition was obtained for
equivariant $K$-theory.

Suppose that we have a short exact sequence of finite groups
\begin{gather*}
1\rightarrow A\rightarrow G\stackrel{\pi}{\rightarrow} Q\rightarrow 1.
\end{gather*}
Let $X$ be a compact and Hausdorff $G$-space on which $A$ acts trivially and fix
$\alpha$ a normalized $2$-cocycle on $G$ with values in $\S^{1}$. Associated to the
cocycle $\alpha$ we have a central extension
\begin{gather*}
1\rightarrow \S^{1}\rightarrow
\widetilde{G}_{\alpha}\rightarrow G\rightarrow 1
\end{gather*}
and the action of $G$ on $X$ can be extended to an action of $\widetilde{G}_{\alpha}$ on $X$
in such a way that the central factor $\S^{1}$ acts trivially. The $\alpha$-twisted
$G$-equivariant $K$-theory of $X$, $^{\alpha}K^{\ast}_{G}(X)$, is~constructed
\mbox{using} $\widetilde{G}_{\alpha}$-equivariant vector bundles on which the central factor
$\S^{1}$ acts by multiplication of~scalars. The main object of study in this work are
the twisted $K$-groups $^{\alpha}K^{\ast}_{G}(X)$.

Representations of the group
$\widetilde{G}_{\alpha}$ on which the central factor $\S^{1}$
acts by multiplication of~sca\-lars are in a one to one correspondence
with $\alpha$-projective representations of $G$. Via this correspondence we can use
the classical tools of projective representations to study the twisted $K$-groups
$^{\alpha}K^{\ast}_{G}(X)$. To this end we show in Section~\ref{Section 2} that,
if $\Irr^{\alpha}(A)$ denotes the set of isomorphism classes of $\alpha$-projective
representations of $A$, then there is an action of $G$ on $\Irr^{\alpha}(A)$
and this action factors through an action of $Q$ on $\Irr^{\alpha}(A)$.
Given an isomorphism class $[\tau]\in \Irr^{\alpha}(A)$ let
$Q_{[\tau]}$ denote the isotropy subgroup at $[\tau]$. Using Lemma~\ref{twistedcocycle}
we show that we can associate to each $[\tau]\in \Irr^{\alpha}(A)$
a normalized $2$-cocycle defined on $Q_{[\tau]}$ with values in $\S^{1}$. This is
precisely the data needed to construct a twisted version of $Q_{[\tau]}$-equivariant
$K$-theory. With this in mind, the main result of this article is the following theorem.

\begin{Th}\label{main}
Suppose that $A$ is a normal subgroup of a finite group $G$. Let $\alpha$
be a normalized $2$-cocycle on $G$ with values in $\S^{1}$ and $X$ a
compact $G$-space on which $A$ acts trivially.
Then there is a natural isomorphism
\begin{align*}
\Phi_{X}\colon \ \phantom{i}^{\alpha}K_{G}^{\ast}(X)&\rightarrow
\bigoplus_{[\tau]\in G\backslash \Irr^{\alpha}(A)}\phantom{i}
^{\beta_{\tau,\alpha}}K^{\ast}_{Q_{[\tau]}}(X),
\\
[E]&\mapsto \bigoplus_{[\tau]\in G\backslash \Irr^{\alpha}(A)}
\big[\Hom_{\widetilde{A}_{\alpha}}(\V_{\tau},E)\big].
\end{align*}
\end{Th}

Explicit formulas for the cocycles $\beta_{\tau,\alpha}$ that appear in the previous theorem
are described in~Section~\ref{Section 2}. Observe that when $\alpha$ is the trivial cocycle
then $\phantom{i}^{\alpha}K_{G}^{\ast}(X)$ agrees
with $K_{G}^{*}(X)$ and the previous decomposition agrees with~\cite[Theorem 3.2]{GomezUribe}
in this case. Therefore Theorem~\ref{main} generalizes~\cite[Theorem 3.2]{GomezUribe}. In fact,
Theorem~\ref{main} is proved using ideas similar to the ones used to prove
\cite[Theorem 3.2]{GomezUribe}.

We remark that the main idea used to prove Theorem~\ref{main} is well known and is often referred
to in the literature as ``Mackey decomposition'' or the ``Mackey machine''. Moreover, similar decomposition
theorems have been obtained using these ideas for the case of equivariant K-theory. To the knowledge of
the authors the first instance where such a decomposition appears is~\cite{Wassermann}.\footnote{The authors would like to thank the editor for pointing out this reference that was not
known previously to them.} In 1981 in his Ph.D.~Thesis Wassermann used the ``Mackey machine'' in
the context of $C^{*}$-algebras to derive a decomposition for equivariant K-theory in the particular
case that a compact Lie group $G$ acts on a locally compact space $X$ with one orbit type.
This means that for every $x\in X$ the isotropy subgroup $G_{x}$ is conjugated to a fixed subgroup $A$
(see~\cite[Theorem~7]{Wassermann}). The decomposition theorem obtained in~\cite{GomezUribe} for equivariant
K-theory is quite similar to the one obtained by Wassermann with the difference that the result derived in~\cite{GomezUribe} works for a finite group~$G$ acting on a compact and Hausdorff space in $X$ such a
way that all the isotropy subgroups~$G_{x}$ contain a fixed group $A$. In particular this means that
the action does not have to~have just one orbit type. Recently, ``Mackey decomposition'' was used in~\cite{ABV}
to~obtain a similar decomposition result for equivariant K-theory in the case of proper actions of a~compact Lie
group~$G$.

Theorem~\ref{main} could also be proved using the work of Freed, Hopkins and
Teleman~\cite{FHT}. However, we have chosen to prove it directly as to obtain an explicit
description of this decomposition. Also, we chose to work with finite groups to obtain explicit
formulas for the cocycles used to~twist equivariant $K$-theory in this decomposition.
Obtaining explicit formulas for such cocycles is one of the main contributions of this
work. Theorem~\ref{main} also holds in general for compact Lie groups, a proof in this context can be
obtained generalizing the work in~\cite{AGU}.

The outline of this article is as follows. In Section~\ref{Section 2} we review some basic definitions
of~pro\-jec\-tive representations that we use throughout this article. Section~\ref{Section 3} is
the main part of the article, Theorem~\ref{main} is proved there. In Section~\ref{Section 4}
we use Theorem~\ref{main} to obtain a formula for the third differential in the
Atiyah--Hirzebruch spectral sequence for $\alpha$-twisted $G$-equivariant $K$-theory
under suitable hypotheses. Finally, in Section~\ref{Section 5} we explore some examples of
Theorem~\ref{main} for the particular case of the dihedral group $D_{2n}$ with $n\ge 1$ an even integer.

\section{Projective representations}\label{Section 2}

In this section we recall some basic definitions and properties of projective representations that
will be used throughout this article.

\subsection{Basic definitions}

\begin{Def}
Let $G$ be a finite group and $V$ a finite dimensional complex
vector space. A~map $\rho\colon G\rightarrow {\rm GL}(V)$ is called a
\textit{projective representation} of $G$ if there exists a
function
\begin{gather*}
\alpha\colon \ G\times G \rightarrow \C^{\ast}
\end{gather*}
such that
\begin{gather}
\rho(g)\rho(h)=\alpha(g,h)\rho(gh)\label{ecu:rep-proj}
\end{gather}
for all $g,h\in G$ and $\rho(1)={\rm Id}_{V}$.
\end{Def}
	
Note that if $\rho$ satisfies equation~(\ref{ecu:rep-proj}) then the function $\alpha$ defines a
$\C^{\ast}$-valued normalized 2-cocycle on $G$; that is, for all $g,h,k\in G$ we have:
\begin{gather*}
\alpha(gh,k)\alpha(g,h)=\alpha(g,hk)\alpha(h,k),
\\
\alpha(g,1)=\alpha(1,g)=1.
\end{gather*}
To stress the dependence of $\rho$ on $V$ and $\alpha$, we shall
often refer to $\rho$ as an $\alpha$-representation of $G$
on the space $V$ or, simply as an $\alpha$-representation of
$G$, if $V$ is not pertinent to the discussion.

\begin{Rem}
If $\alpha$ is the trivial cocycle; that is, if $\alpha(g,h)=1$ for all $g,h\in G$, then
$\alpha$-representations of $G$ are simply ordinary representations of $G$.
\end{Rem}

\begin{Def}
Suppose that $\rho_{i}\colon G\rightarrow {\rm GL}(V_{i})$ with $(i=1,2)$ are two $\alpha$-representations.
\begin{enumerate}\itemsep=0pt
\item A linear map $\varphi\colon V_{1}\to V_{2}$ is said to be a map of projective representations or a
$G$-map if~for any $g\in G$ and any $v\in V_{1}$ we have $\varphi(\rho_{1}(g)v)=\rho_{2}(g)\varphi(v)$.
We write $\Hom_{G}(V_{1},V_{2})$ for the set of $G$-morphisms from $V_{1}$ to $V_{2}$.

\item The $\alpha$-representations $V_{1}$ and $V_{2}$
are said to be \textit{linearly equivalent} or \textit{isomorphic} if
there exists a $G$-map $f\colon V_{1}\to V_{2}$ that is an isomorphism of vector spaces.
In other words, $V_{1}$ and $V_{2}$ are isomorphic if there is a vector space
isomorphism $f\colon V_{1}\rightarrow V_{2}$ such that $\rho_{2}(g)=f\rho_{1}(g)f^{-1}$ for all $g\in G$.
\end{enumerate}
\end{Def}

Just as in the case of the ordinary theory of representations of groups
we have similar notions for projective representations such as irreducible representations
and unitary representations. Given a $\C^{\ast}$-valued normalized 2-cocycle $\alpha$ on
$G$ we denote by $\Irr^{\alpha}(G)$ the set of isomorphism classes of
complex irreducible $\alpha$-representations of $G$.
If $\rho\colon G\to {\rm GL}(V)$ is an irreducible $\alpha$-representation of $G$ then $[\rho]\in \Irr^{\alpha}(G)$
denotes the corresponding isomorphism class. Observe that the direct sum of two $\alpha$-representations
is also an $\alpha$-representation. Thus we can form the monoid of isomorphism classes of $\alpha$-representations.
The $\alpha$-twisted representation group of $G$, denoted by $R^{\alpha}(G)$,
is defined as the associated Grothendieck group. As an abelian group $R^{\alpha}(G)$ is a free abelian
group with one generator for each element in $\Irr^{\alpha}(G)$.
We remark that the classical results of representation theory
such as complete reducibility and Schur's lemma also hold for the case of projective representations.
We refer the reader to~\cite{Karpilovsky2} for
the basic theory of~projective representations.

\begin{Ex}\label{ej,D2n}
Consider the dihedral group $D_{2n}$ of order $2n$ defined by
\begin{gather*}
D_{2n}=\big\langle a,b\mid a^{n}=b^{2}=1,\, bab=a^{-1}\big\rangle.
\end{gather*}
For such groups we have
\begin{gather*}
H^{2}\big(D_{2n},\S^{1}\big)\cong
\begin{cases}
1 &\text{if}\quad n \text{ is odd},
\\
\Z/2 & \text{if}\quad n \text{ is even}.
\end{cases}
\end{gather*}
In this example we only consider the case where $n$ is even as otherwise we will obtain
usual representations. Let $n\geq 2$ be an even integer, $\epsilon$ a primitive $n$-th root of
unity in $\C$ and let
\begin{gather*}
\alpha\colon\ D_{2n}\times D_{2n}\rightarrow \S^{1}
\end{gather*}
be the function defined by
\begin{gather*}
\alpha\big(a^{j},a^{k}b^{l}\big)=1 \qquad \text{and} \qquad
\alpha\big(a^{j}b,a^{k}b^{l}\big)=\epsilon^{k}\quad \text{for}\quad
0\leq j,k\leq n-1 \quad \text{and}\quad l=0,1.
\end{gather*}
The function $\alpha$ defines a normalized $2$-cocycle on $D_{2n}$ with values in $\S^{1}$ whose
corresponding cohomology class is a generator in $H^{2}\big(D_{2n},\S^{1}\big)\cong \Z/2$.
For each $i\in\{1,2,\dots,n/2\}$ put
\begin{gather*}
A_{i}=\begin{pmatrix}
\epsilon^{i} & 0 \\
0 & \epsilon^{1-i}
\end{pmatrix}
\qquad \text{and} \qquad
B_{i}=\begin{pmatrix}
0 & 1\\
1 &0
\end{pmatrix}\!.
\end{gather*}
Consider the map
\begin{gather*}
\tau_{i}\colon\ D_{2n}\rightarrow {\rm GL}_{2}(\C)
\end{gather*}
defined by
\begin{gather*}
\tau_{i}\big(a^{k}b^{l}\big)=A_{i}^{k}B_{i}^{l} \qquad \text{for} \quad 0\leq k\leq n-1 \quad
\text{and}\quad l=0,1.
\end{gather*}
These assignments determine the irreducible, non-equivalent $\alpha$-representations of $D_{2n}$
so that $\Irr^{\alpha}(D_{2n})=\big\{[\tau_{1}],\dots, [\tau_{n/2}]\big\}$ and as an abelian group we
have an isomorphism $R^{\alpha}(D_{2n})\cong \bigoplus_{i=1}^{n/2}\Z[\tau_{i}]$
(see for example~\cite[Chapter 5, Theorem 7.1]{Karpilovsky3}).
\end{Ex}

A key feature of projective representations is that we also have the following canonical
decomposition whose proof can be obtained in a similar way as in the case of regular
representations.

\begin{Th}[canonical decomposition]\label{can,decom}
Suppose that $\alpha$ is a
normalized $2$-cocycle of a finite group $G$ with values in $\S^{1}$. Let $W$ be a
finite-dimensional $\alpha$-representation. Then the assignment
\begin{align*}
\gamma\colon\ \bigoplus_{[\rho]\in \Irr^{\alpha}(G)}V_{\rho}\otimes \Hom_{G}(V_{\rho},W)
&\rightarrow W,
\\
v\otimes f&\mapsto f(v)
\end{align*}
defines an isomorphism of $\alpha$-representations.
\end{Th}

Suppose now that $\alpha$ is a normalized 2-cocycle on $G$ with values in $\S^{1}$.
We can associate to $\alpha$ a central extension of $G$ by $\S^{1}$ in the following way.
As a set define
\begin{gather*}
\widetilde{G}_{\alpha}=\big\{(g,z)\mid g\in G, z\in \S^{1}\big\}.
\end{gather*}
The product structure in $\widetilde{G}_{\alpha}$ is given by
the assignment
\begin{gather*}
(g_{1},z_{1})(g_{2},z_{2}):=\left(g_{1}g_{2},\alpha(g_{1},g_{2})z_{1}z_{2}\right).
\end{gather*}
This way $\widetilde{G}_{\alpha}$ is a compact Lie group that fits
into a central extension
\begin{gather*}
1\rightarrow \S^{1}\rightarrow
\widetilde{G}_{\alpha}\rightarrow G\rightarrow 1.
\end{gather*}
Let $\rho\colon G\to {\rm GL}(V)$ be an $\alpha$-representation of $G$. If we define
$\tilde{\rho}\colon \widetilde{G}_{\alpha}\to {\rm GL}(V)$
by $\tilde{\rho}(g,z)=z\rho(g)$ then $\tilde{\rho}$ defines a representation of
$\widetilde{G}_{\alpha}$ on which the central factor $\S^{1}$ acts by multiplication of~scalars.
Conversely, if $\tilde{\rho}\colon \widetilde{G}_{\alpha}\to {\rm GL}(V)$ is a representation of
$\widetilde{G}_{\alpha}$ on which the central factor~$\S^{1}$ acts by multiplication of scalars
then the function $\rho\colon G\to {\rm GL}(V)$ given by $\rho(g)=\tilde{\rho}(g,1)$ defines an
$\alpha$-representation of $G$. The above assignment defines a one to one correspondence between
$\alpha$-representations of $G$ and representations of $\widetilde{G}_{\alpha}$ on which the
central factor $\S^{1}$ acts by multiplication of scalars. Via this correspondence we will
switch back and forth between $\alpha$-representations of $G$ and representations of
$\widetilde{G}_{\alpha}$ on which the central factor $\S^{1}$ acts by multiplication of scalars
without explicitly mentioning it. In addition, if $V_{1}$ and $V_{2}$ are two
$\alpha$-repre\-sentations of $G$ then having a map of projective representations
$f\colon V_{1}\to V_{2}$ is equivalent to~having a linear map
$f\colon V_{1}\to V_{2}$ that is $\widetilde{G}_{\alpha}$-equivariant. With this correspondence
in mind we~will also identify $\Hom_{G}(V_{1},V_{2})$ with
$\Hom_{\widetilde{G}_{\alpha}}(V_{1},V_{2})$ without explicitly mentioning it.

\subsection{Cocycles and projective representations}\label{section cocycles}

Suppose now that $A$ is a normal subgroup of a finite group $G$ so that we have a short exact
sequence
\begin{gather*}
1\rightarrow A\rightarrow G\stackrel{\pi}{\rightarrow} Q\rightarrow 1.
\end{gather*}
Assume that $\alpha$ is a normalized $2$-cocycle of $G$ with values in $\S^{1}$,
by restriction we can also see $\alpha$ as a cocycle defined on $A$.
Let us define an action of $G$ on the set $\Irr^{\alpha}(A)$ in the following way.
Given $\rho\colon A\rightarrow U(V_{\rho})$ an $\alpha$-representation and $g\in G$ we define
$g\cdot \rho\colon A\rightarrow U(V_{\rho})$ so that if~$a\in A$ we have:
\begin{gather}\label{def:act}
g\cdot \rho(a)=\alpha(g^{-1}a,g)\alpha(g,g^{-1}a)^{-1}\rho(g^{-1}ag)\in U(V_{\rho}).
\end{gather}

\begin{Pro}
The above assignment defines a left action of $G$ on $\Irr^{\alpha}(A)$. Furthermore, for all
$b\in A$, we have that $b\cdot\rho\cong\rho$ so that the action of $G$ on $\Irr^{\alpha}(A)$
factors to an action of~$Q=G/A$ on $\Irr^{\alpha}(A)$.
\end{Pro}

\begin{proof}
First we show that $g\cdot \rho$ is an $\alpha$-representation
of $A$. Indeed, for all $g\in G$ and $a,b \in A$ we have
\begin{align*}
(g\cdot \rho(a))(g\cdot\rho(b))&=
\big(\alpha\big(g^{-1}a,g\big)\alpha\big(g,g^{-1}a\big)^{-1}\rho\big(g^{-1}ag\big)\big)
\big(\alpha\big(g^{-1}b,g\big)\alpha\big(g,g^{-1}b\big)^{-1}\rho\big(g^{-1}bg\big)\big)
\\
&=\alpha\big(g^{-1}a,g\big)\alpha\big(g,g^{-1}a\big)^{-1}\alpha\big(g^{-1}b,g\big)\alpha\big(g,g^{-1}b\big)^{-1}
\alpha\big(g^{-1}ag,g^{-1}bg\big)
\\
&\phantom{=}\,\times\rho\big(g^{-1}abg\big)
=\alpha\big(g,g^{-1}a\big)^{-1}\alpha\big(g^{-1}ab,g\big)\alpha\big(g^{-1}a,b\big)\rho\big(g^{-1}abg\big)
\\
&=\alpha(a,b)\alpha\big(g^{-1}ab,g\big)\alpha\big(g,g^{-1}ab\big)^{-1}\rho\big(g^{-1}abg\big)\\
&=\alpha(a,b) (g\cdot \rho(ab) ).
\end{align*}
The above equalities are obtained using the 2-cocycle equation for $\alpha$.
Now, we show that this definition satisfies the axioms of an action. If
$\rho\colon A\rightarrow U(V_{\rho})$ is an $\alpha$-representation,
as $\alpha$ is a~normalized cocycle,
\begin{gather*}
1\cdot \rho(a)=\alpha(a,1)\alpha(1,a)^{-1}\rho(a)=\rho(a).
\end{gather*}
Moreover, given $g,h\in G$ and $a\in A$, we have
\begin{align*}
g\cdot(h\cdot \rho)(a)&=\alpha\big(g^{-1}a,g\big)\alpha\big(g,g^{-1}a\big)^{-1}(h\cdot
\rho)\big(g^{-1}ag\big)
\\
&=\alpha\big(g^{-1}a,g\big)\alpha\big(g,g^{-1}a\big)^{-1}\alpha\big(h^{-1}g^{-1}ag,h\big)
\alpha\big(h,h^{-1}g^{-1}ag\big)^{-1}\rho\big(h^{-1}g^{-1}agh\big)
\end{align*}
and
\begin{align*}
(gh)\cdot \rho(a)&=\alpha\big(h^{-1}g^{-1}a,gh\big)\alpha\big(gh,h^{-1}g^{-1}a\big)^{-1}\rho\big((gh)^{-1}a(gh)\big)
\\
&=\alpha\big(h^{-1}g^{-1}a,gh\big)\alpha\big(gh,h^{-1}g^{-1}a\big)^{-1}\rho\big(h^{-1}g^{-1}agh\big).
\end{align*}
Manipulating the cocycle equation for $\alpha$ it can be proved that
\begin{align*}
\alpha\big(g^{-1}a,g\big)&\alpha\big(g,g^{-1}a\big)^{-1}\alpha\big(h^{-1}g^{-1}ag,h\big)\alpha\big(h,h^{-1}g^{-1}ag\big)^{-1}
\\
&=\alpha\big(h^{-1}g^{-1}a,gh\big)\alpha\big(gh,h^{-1}g^{-1}a\big)^{-1}.
\end{align*}
This implies that for all $a\in A$
\begin{gather*}
g\cdot(h\cdot \rho)(a)=(gh)\cdot \rho(a).
\end{gather*}
Finally, for $a, b\in A$ expanding and using the cocycle equation we obtain
\begin{gather*}
b\cdot\rho(a)=\rho(b)^{-1}\rho(a)\rho(b)
\end{gather*}
and thus $b\cdot\rho\cong\rho$ as $\alpha$-representations.
\end{proof}

\begin{Rem}The action of $G$ on the set $\Irr^{\alpha}(A)$ given by equation~(\ref{def:act}) can be
described in an alternative way as follows. Suppose that $\alpha$ is a normalized $2$-cocycle of
$G$ with values in~$\S^{1}$ and let $\widetilde{G}_{\alpha}$ be the central extension corresponding
to the cocycle $\alpha$. If~$\rho\colon A\rightarrow U(V_{\rho})$ is an~$\alpha$-representation of $A$
then as explained before we have an associated representation \mbox{$\tilde{\rho}\colon \tilde{A}_{\alpha}\to U(V_{\rho})$}.
Given $(g,z)\in \widetilde{G}_{\alpha}$ we obtain the $\tilde{A}_{\alpha}$-representation
$(g,z)\cdot \tilde{\rho}$ defined by
\begin{gather*}
(g,z)\cdot \tilde{\rho}(a,w)=\tilde{\rho}\big((g,z)^{-1}(a,w)(g,z)\big).
\end{gather*}
This is a well-defined $\tilde{A}_{\alpha}$-representation such that the central factor
$\S^{1}$ acts by multiplication of scalars. Interpreting this $\tilde{A}_{\alpha}$-representation as an
$\alpha$-representation and using the cocycle identities we obtain equation~(\ref{def:act}). Explicitly,
for $g\in G$ and $a\in A$ we have
\begin{align*}
g\cdot \rho(a):\!&=(g,1)\cdot \tilde{\rho}(a,1)= \tilde{\rho}\big(g^{-1}ag,\alpha\big(g^{-1}a,g\big)\alpha\big(g,g^{-1}a\big)^{-1}\big)
\\
&=\alpha\big(g^{-1}a,g\big)\alpha\big(g,g^{-1}a\big)^{-1} \rho\big(g^{-1}ag\big).
\end{align*}
\end{Rem}


As above assume that $A$ is a normal subgroup of a group $G$ and $Q=G/A$. Fix an assignment
$\sigma \colon Q \rightarrow G$ such that $\pi(\sigma(q)) = q$ for all $q \in Q$ with $\sigma(1)
= 1$. We remark that the map $\sigma$ is only a set theoretical map so in particular it does
not necessarily have to be a group homomorphism.

Suppose that $\rho\colon A \rightarrow U(V_{\rho})$ is a complex irreducible $\alpha$-representation
with the property that $g\cdot \rho$ is isomorphic to $\rho$ for every $g\in G$
(under the action defined in equation~(\ref{def:act})). Under this assumption,
as $\sigma(q)\cdot \rho\cong\rho$
we can find an element $M_{q}\in
U(V_{\rho})$ for each $q \in Q$ such that
\begin{gather*}
	\sigma(q)\cdot \rho(a)=M_{q}^{-1}\rho(a)M_{q}.
\end{gather*}
This means that
\begin{gather*}
\alpha\big(\sigma(q)^{-1}a,\sigma(q)\big)\alpha\big(\sigma(q),\sigma(q)^{-1}a\big)^{-1}
\rho\big(\sigma(q)^{-1}a\sigma(q)\big)=M_{q}^{-1}\rho(a)M_{q}
\end{gather*}
for all $a\in A$. We can choose $M_{1} = 1$ as $\sigma(1)=1$ and $\sigma(1)\cdot\rho=\rho$.
Let $\widetilde{A}_{\alpha}$ be the central extension of $A$ by $\S^{1}$
associated to the cocycle $\alpha$ and $\tilde{\rho}\colon \widetilde{A}_{\alpha}\to U(V_{\rho})$ the
corresponding representation. Remember that $\tilde{\rho}(a,z)=z\rho(a)$ so we can define
\begin{gather*}
\sigma(q)\cdot \tilde{\rho}(a,z):=z (\sigma(q)\cdot \rho(a)).
\end{gather*}
Therefore
\begin{gather*}
\sigma(q)\cdot \tilde{\rho}(a,z)=z\big(M_{q}^{-1}\rho(a)M_{q}\big)
=M_{q}^{-1}z\rho(a)M_{q}=M_{q}^{-1}\tilde{\rho}(a,z)M_{q}.
\end{gather*}
Define $\chi\colon Q\times Q \rightarrow A$ by the equation
\begin{gather*}
\chi(q_{1},q_{2})=\sigma(q_{1}q_{2})^{-1}\sigma(q_{1})\sigma(q_{2})
\end{gather*}
Note that $\chi(q_{1},q_{2})$ belongs to $A$ since
$\pi(\chi(q_{1},q_{2}))=1$ and the map $\chi$ is normalized in
the sense that $\chi(q_{1},q_{2})=1$ whenever $q_{1}=1$ or
$q_{2}=1$. In addition, define $\tau\colon Q\times Q\rightarrow \S^{1}$ by
\begin{gather*}
\tau(q_{1},q_{2})=\alpha\big(\sigma(q_{1}q_{2})^{-1},\sigma(q_{1})\sigma(q_{2})\big)
\alpha\big(\sigma(q_{1}q_{2}),\sigma(q_{1}q_{2})^{-1}\big)^{-1}\alpha(\sigma(q_{1}),\sigma(q_{2}))\nonumber
\\ \hphantom{\tau(q_{1},q_{2})}
{}=\alpha\big(\sigma(q_{1}q_{2}),\chi(q_{1},q_{2})\big)^{-1}\alpha(\sigma(q_{1}),\sigma(q_{2})).
\end{gather*}
Now, for $q_{1}$, $q_{2}$ in $Q$ we notice that the element
\begin{gather*}
\tilde{\rho}(\chi(q_{1},q_{2}),\tau(q_{1},q_{2}))M_{q_{2}}^{-1}M_{q_{1}}^{-1}M_{q_{1}q_{2}}
=\tau(q_{1},q_{2})\rho(\chi(q_{1},q_{2}))M_{q_{2}}^{-1}M_{q_{1}}^{-1}
M_{q_{1}q_{2}}
\end{gather*}
belongs to the center $Z(U(V_{\rho}))\cong \S^{1}$. Define the map
$\beta_{\rho,\alpha}\colon Q\times Q \rightarrow
\S^{1}$ by the equation
\begin{gather}
\label{defQbeta}
\beta_{\rho,\alpha}(q_{1},q_{2}):
=\tilde{\rho}(\chi(q_{1},q_{2}),\tau(q_{1},q_{2}))M_{q_{2}}^{-1}M_{q_{1}}^{-1}M_{q_{1}q_{2}}.
\end{gather}

\begin{Lem}\label{twistedcocycle}
If $\rho\colon A \rightarrow U(V_{\rho})$ is a complex
irreducible $\alpha$-representation such that $g\cdot \rho\cong \rho$ for every
$g\in G$ then the map $\beta_{\rho,\alpha}\colon Q\times Q\rightarrow \S^{1}$ defines a
normalized $2$-cocycle on $Q$ with values in~$\S^{1}$.
\end{Lem}

\begin{proof}
If either $q_{1}=1$ or $q_{2}=1$ we have that $\tau(q_{1},q_{2})=1$ as $\alpha$ is normalized.
Also, as $\chi$ is normalized we have $\chi(q_{1},q_{2})=1$. Since we are choosing
$\sigma(1)=1$ and $M_{1}=1$ it follows that $\beta_{\rho,\alpha}(q_{1},q_{2})=1$ if either
$q_{1}=1$ or $q_{2}=1$. Therefore either $\beta_{\rho,\alpha}$ is normalized. To finish we need
to prove that for every $q_{1},q_{2},q_{3}\in Q$ we have
\begin{gather*}
\beta_{\rho,\alpha}(q_{1},q_{2}q_{3})\beta_{\rho,\alpha}(q_{2},q_{3})
=\beta_{\rho,\alpha}(q_{1}q_{2},q_{3})\beta_{\rho,\alpha}(q_{1},q_{2}).
\end{gather*}
To see this note that as $\beta_{\rho,\alpha}(q_{1},q_{2})$ belongs to $\S^{1}$ we
have that
\begin{gather*}
M_{q_{1}q_{2}}
=\beta_{\rho,\alpha}(q_{1},q_{2})M_{q_{1}}M_{q_{2}}
\tilde{\rho}(\chi(q_{1},q_{2}),\tau(q_{1},q_{2}))^{-1}.
\end{gather*}
Therefore, for $q_1, q_2$ and $q_3$ in $Q$ we have,
\begin{align*}
M_{q_{1}q_{2}q_{3}}&=M_{q_{1}(q_{2}q_{3})}
=\beta_{\rho,\alpha}(q_{1},q_{2}q_{3})M_{q_{1}}
M_{q_{2}q_{3}}\tilde{\rho}(\chi(q_{1},q_{2}q_{3}),\tau(q_{1},q_{2}q_{3}))^{-1}
\\
&=\beta_{\rho,\alpha}(q_{1},q_{2}q_{3})M_{q_{1}}
\beta_{\rho,\alpha}(q_{2},q_{3})M_{q_{2}}M_{q_{3}}
\tilde{\rho}(\chi(q_{2},q_{3}),\tau(q_{2},q_{3}))^{-1}
\\
&\phantom{=}\times\tilde{\rho}(\chi(q_{1},q_{2}q_{3}),
\tau(q_{1},q_{2}q_{3}))^{-1}
=\beta_{\rho,\alpha}(q_{1},q_{2}q_{3})\beta_{\rho,\alpha}(q_{2},q_{3})M_{q_{1}}M_{q_{2}}M_{q_{3}}
\\
&\phantom{=}\times\big[\tilde{\rho}(\chi(q_{1},q_{2}q_{3}),\tau(q_{1},q_{2}q_{3}))
\tilde{\rho}(\chi(q_{2},q_{3}),\tau(q_{2},q_{3}))\big]^{-1}.
\end{align*}
In a similar way, writing $M_{q_{1}q_{2}q_{3}}=M_{(q_{1}q_{2})q_{3}}$ and expanding we obtain
\begin{align*}
M_{q_{1}q_{2}q_{3}}&=\beta_{\rho,\alpha}(q_{1}q_{2},q_{3})
\beta_{\rho,\alpha}(q_{1},q_{2})M_{q_{1}}M_{q_{2}}M_{q_{3}}
\\
&\phantom{=}\times\big[\tilde{\rho}(\chi(q_{1}q_{2},q_{3}),\tau(q_{1}q_{2},q_{3}))
\tilde{\rho}(\chi(q_{1},q_{2}),\tau(q_{1},q_{2}))\big]^{-1}.
\end{align*}	
This implies that
\begin{gather*}
\beta_{\rho,\alpha}(q_{1},q_{2}q_{3})\beta_{\rho,\alpha}(q_{2},q_{3})
=\beta_{\rho,\alpha}(q_{1}q_{2},q_{3})\beta_{\rho,\alpha}(q_{1},q_{2})
\end{gather*}
as we wanted to prove.
\end{proof}

%
%

\section[Decomposition of twisted equivariant K-theory]
{Decomposition of twisted equivariant $\boldsymbol K$-theory}\label{Section 3}

In this section we use the canonical decomposition of vector bundles to show that,
under some hypothesis, the $\alpha$-twisted equivariant $K$-theory $^{\alpha}K^{\ast}_{G}(X)$ of a
$G$-space $X$ can be decomposed as a direct sum of twisted equivariant K-theories parametrized
by the orbits of the action of $G$ on~$\Irr^{\alpha}(A)$ constructed on the previous section.

We start by recalling the definition of $\alpha$-twisted equivariant $K$-theory that we will use.
We~fol\-low the treatment used in~\cite[Section 7.2]{AdemRuan}.
Assume that $G$ is a finite group acting on a compact and Hausdorff space $X$.
Let $\alpha$ be a normalized 2-cocycle on $G$ with values in $\S^{1}$. Consider
\begin{gather*}
1\rightarrow \S^{1}\rightarrow \widetilde{G}_{\alpha}\rightarrow G\rightarrow 1
\end{gather*}
the corresponding central extension. Observe that the action of $G$ on $X$ can be
extended to an~action of $\widetilde{G}_{\alpha}$ in such a way that the central factor $\S^{1}$
acts trivially.

\begin{Def}\label{usualdeftwisted}
The $\alpha$-twisted $G$-equivariant $K$-theory of $X$, denoted by $^{\alpha}K_{G}^{0}(X)$,
is defined as the Grothendieck group of the set of isomorphism classes of
$\widetilde{G}_{\alpha}$-equivariant vector bundles over $X$ on which
$\S^{1}$ acts by multiplication of scalars on the fibers. For $n>0$ the twisted groups
$^{\alpha}K_{G}^{n}(X)$ are defined as
$^{\alpha}\tilde{K}_{G}^{0}(\Sigma^{n}X_{+})$, where as usual $X_{+}$
denotes the space $X$ with an added base point.
\end{Def}

To obtain the desired decomposition of twisted equivariant $K$-groups we are going to construct an equivalent
formulation for such twisted K-groups following the work in~\cite{GomezUribe}. For this suppose
that we have a short exact sequence of finite groups
\begin{gather*}
1\rightarrow A\rightarrow G\stackrel{\pi}{\rightarrow} Q\rightarrow 1.
\end{gather*}
Assume that $G$ acts on a compact and Hausdorff space $X$ in such a way that $A$ acts trivially.
Fix $\alpha\in Z^{2}\big(G,\S^{1}\big)$ a normalized 2-cocycle. Notice that by restriction we can
see any $\widetilde{G}_{\alpha}$-equivariant vector bundle over $X$ as an
$\widetilde{A}_{\alpha}$-equivariant vector bundle, where $\widetilde{A}_{\alpha}$ denotes the
central extension associated to the cocycle $\alpha$ seen as a cocycle defined on $A$.
Also, recall that associated to an~$\alpha$-representation $\rho\colon A \rightarrow U(V_{\rho})$ we have a representation
$\tilde{\rho}\colon \widetilde{A}_{\alpha} \rightarrow U(V_{\rho})$.

{\sloppy\begin{Def}
Suppose that $\rho\colon A \rightarrow U(V_{\rho})$ is a complex irreducible $\alpha$-representation.
A~$(\widetilde{G}_{\alpha}, \rho)$-equivariant vector bundle over $X$ is a
$\widetilde{G}_{\alpha}$-vector bundle on which the central factor $\S^{1}$ acts
by multiplication of scalars on $E$ and the map
\begin{align*}
\gamma\colon\ \V_{\rho}\otimes \Hom_{\widetilde{A}_{\alpha}}(\V_{\rho},E)&\rightarrow E,
\\
v\otimes f &\mapsto f(v)
\end{align*}
is an isomorphism of $\widetilde{A}_{\alpha}$-vector bundles on which $\S^{1}$ acts
by multiplication of scalars.
\end{Def}

}

In the above definition, if $\rho$ is an $\alpha$-representation of $A$ then $\V_{\rho}$
denotes the trivial $\widetilde{A}_{\alpha}$-vector bundle $\pi_{1}\colon X\times V_{\rho}\to X$.
Observe that if $p\colon E\to X$ is a $\widetilde{G}_{\alpha}$-vector bundle on which the central
factor $\S^{1}$ acts by multiplication then, as we are assuming that $A$ acts trivially
on $X$, it follows that for every $x \in X$ the fiber $E_{x}$ can be seen as a representation of
$\widetilde{A}_{\alpha}$ on which the central factor acts by scalar multiplication. Thus for every
$x \in X$ the fiber $E_{x}$ can be seen as an $\alpha$-representation of $A$.
With this point of view, a $\big(\widetilde{G}_{\alpha}, \rho\big)$-equivariant vector bundle is a
$\widetilde{G}_\alpha$-equivariant vector bundle $p\colon E\rightarrow X$ such that for every
$x \in X$ the fiber $E_{x}$ is an~$\alpha$-representation of $A$ isomorphic to a direct sum of
the $\alpha$-representation $\rho$. Let $\text{Vect}_{\widetilde{G}_{\alpha}, \rho}(X)$
denote the set of isomorphism classes of $\big(\widetilde{G}_{\alpha}, \rho\big)$-equivariant vector
bundles, where two $\big(\widetilde{G}_{\alpha}, \rho\big)$-equivariant vector bundles are isomorphic
if they are isomorphic as $\widetilde{G}_{\alpha}$-vector bundles. Notice that if
$E_{1}$ and $E_{2}$ are two $\big(\widetilde{G}_{\alpha}, \rho\big)$-equivariant vector bundles
then so is $E_{1}\oplus E_{2}$. Therefore $\text{Vect}_{\widetilde{G}_{\alpha}, \rho}(X)$
is a semigroup. Following~\cite[Definition~2.2]{GomezUribe} we have the next definition.

\begin{Def}\label{def,grhoalfa}
Assume that $G$ acts on a compact space $X$ in such a way that $A$ acts trivially on $X$
and let $\alpha$ be a normalized 2-cocycle $\alpha\in Z^{2}\big(G,\S^{1}\big)$. We define
$K^{0}_{\widetilde{G}_{\alpha}, \rho}(X)$, the $\big(\widetilde{G}_{\alpha}, \rho\big)$-equivariant
$K$-theory of $X$, as the Grothendieck construction applied to
$\text{Vect}_{\widetilde{G}_{\alpha}, \rho}(X)$. For $n>0$ the group
$K^{n}_{\widetilde{G}_{\alpha}, \rho}(X)$ is defined as
$\tilde{K}^{0}_{\widetilde{G}_{\alpha}, \rho}(\Sigma^{n}X_{+})$.
\end{Def}

As our next step we show that the previous definition can be described using the usual definition
of twisted equivariant $K$-theory provided in Definition~\ref{usualdeftwisted}.
For this suppose that $\alpha$ is a normalized 2-cocycle of $G$ with values in $\S^{1}$.
As above assume that $A$ is a normal subgroup of $G$ and let $Q=G/A$. Let $\rho$ be an
irreducible $\alpha$-representation such that $g\cdot\rho\cong\rho$ for all $g \in G$. Fix
a set theoretical section $\sigma\colon Q\to G$ such that $\sigma(1)=1$ as in the previous section.
We can extend $\sigma$ to obtain a map $\tilde{\sigma}\colon Q\to \widetilde{G}_{\alpha}$ by defining
$\tilde{\sigma}(q)=(\sigma(q),1)\in \widetilde{G}_{\alpha}$. Let $\beta_{\rho,\alpha}$ be the
$2$-cocycle defined on $Q$ with values in $\S^{1}$ constructed
in equation~(\ref{defQbeta}). With this cocycle we can consider the central extension
\begin{gather*}
1\rightarrow \S^{1}\rightarrow \widetilde{Q}_{\beta_{\rho,\alpha}}
\rightarrow Q\rightarrow 1.
\end{gather*}
With this in mind we have the following generalization of~\cite[Theorem 2.1]{GomezUribe}.

\begin{Th}\label{teo,bij}
Suppose that $\rho$ is an irreducible $\alpha$-representation such that $g\cdot\rho\cong\rho$
for all $g \in G$. Let $X$ be a $G$-space such that $A$ acts trivially on $X$.
If $p\colon E\rightarrow X$ is a $\big(\widetilde{G}_{\alpha}, \rho\big)$-equivariant vector bundle, then
$\Hom_{\widetilde{A}_{\alpha}}(\V_{\rho},E)$ has the
structure of a $\widetilde{Q}_{\beta_{\rho,\alpha}}$-vector bundle on which
the central factor~$\S^{1}$ acts by multiplication of
scalars. Moreover, the assignment
\begin{gather*}
[E]\rightarrow
[\Hom_{\widetilde{A}_{\alpha}}(\V_{\rho},E)]
\end{gather*}
is a natural one to one correspondence between isomorphism
classes of $\big(\widetilde{G}_{\alpha},\rho\big)$-equivariant
vector bundles over $X$ and isomorphism classes of
$\widetilde{Q}_{\beta_{\rho,\alpha}}$-equivariant vector bundles over $X$
for which the central factor $\S^{1}$ acts by multiplication of scalars.
\end{Th}
\begin{proof} We are only going to provide a sketch of the proof as it follows the same
steps used in the proof of~\cite[Theorem 2.1]{GomezUribe}.

Suppose that $p\colon E \rightarrow X$ is a $\widetilde{G}_{\alpha}$-vector bundle. We give
$\Hom_{\widetilde{A}_{\alpha}}(\V_{\rho},E)$ an action of
	$\widetilde{Q}_{\beta_{\rho,\alpha}}$. Given $f \in
\Hom_{\widetilde{A}_{\alpha}}(\V_{\rho},E)_{x}$ and
$q\in Q$ we define $q \bullet f \in
\Hom_{\widetilde{A}_{\alpha}}(\V_{\rho},E)_{q\cdot x}$ by
\begin{gather*}
(q\bullet f)(v)=\tilde{\sigma}(q)\cdot f\big(M_{q}^{-1}v\big)
=(\sigma(q),1) \cdot f\big(M_{q}^{-1}v\big),
\end{gather*}
where $M_{q} \in U(V_{\rho})$ is the element chosen in the Section~\ref{section cocycles}.
It is easy to see that $q \bullet f$ is $\widetilde{A}_{\alpha}$-equivariant. Also,
if $q_{1}, q_{2}\in Q$ then $q_{1}\bullet
(q_{2} \bullet f)$ is such that for $v \in V_{\rho}$
\begin{align*}
q_{1}\bullet(q_{2}\bullet f)(v)&=\tilde{\sigma}(q_{1})
(q_{2}\bullet f)\big(M_{q_{1}}^{-1}v\big)=\tilde{\sigma}(q_{1})\tilde{\sigma}(q_{2})
f\big(M_{q_{2}}^{-1}M_{q_{1}}^{-1}v\big)\\
&=(\sigma(q_{1}q_{2}),1)(\chi(q_{1},q_{2}),\tau(q_{1},q_{2}))
f\big(M_{q_{2}}^{-1}M_{q_{1}}^{-1}v\big)\\
&=\tilde{\sigma}(q_{1}q_{2})f(\tilde{\rho}\big((\chi(q_{1},q_{2}),\tau(q_{1},q_{2}))
M_{q_{2}}^{-1}M_{q_{1}}^{-1}v\big)
\\
&=\tilde{\sigma}(q_{1}q_{2})
f\big(\beta_{\rho,\alpha}(q_{1},q_{2})M_{q_{1}q_{2}}^{-1}v\big)=\beta_{\rho,\alpha}(q_{1},q_{2}) \tilde{\sigma}(q_{1}q_{2})
f\big(M_{q_{1}q_{2}}^{-1}v\big)\\
&=\beta_{\rho,\alpha}(q_{1},q_{2})(q_{1}q_{2}\bullet f(v)).
\end{align*}
We conclude that
\begin{gather*}
q_{1}\bullet(q_{2}\bullet
f)(v)=\beta_{\rho,\alpha}(q_{1},q_{2})(q_{1}q_{2}\bullet f(v)).
\end{gather*}
The last equation allows us to define an action of
$\widetilde{Q}_{\beta_{\rho,\alpha}}$ on
$\Hom_{\widetilde{A}_{\alpha}}(\V_{\rho},E)$ as follows. If
$(q, z) \in Q_{\beta_{\rho,\alpha}}$ and $f \in
\Hom_{\widetilde{A}_{\alpha}}(\V_{\rho},E)_{x}$ define
\begin{gather*}
(q,z)\cdot f:=z(q\bullet f).
\end{gather*}
Thus if $v\in V_{\rho}$ then
\begin{gather*}
((q,z)\cdot f)(v)=z\big(\tilde{\sigma}(q)\cdot f\big(M_{q}^{-1}v\big)\big)=
\tilde{\sigma}(q)\cdot \big(zf\big(M_{q}^{-1}v\big)\big).
\end{gather*}
Unraveling the definitions, it is easy to see that this way
$\Hom_{\widetilde{A}_{\alpha}}(\V_{\rho},E)$ has the structure of
a~$\widetilde{Q}_{\beta_{\rho,\alpha}}$-equivariant vector bundle such that
the central factor $\S^{1}$ acts by multiplication of scalars.

Suppose now that $p \colon F \rightarrow X$ is a
$\widetilde{Q}_{\beta_{\rho,\alpha}}$-equivariant vector bundle over $X$
for which the central~$\S^{1}$ acts by multiplication of
scalars. Given $(g,z)\in \widetilde{G}_{\alpha}$, $f\in F_{x}$ and $v\in
V_{\rho}$ define
\begin{gather}
(g,z)\cdot (v\otimes f):=M_{\pi(g)}\tilde{\rho}\big(\tilde{\sigma}(\pi(g))^{-1}(g,z)\big)v\otimes
((\pi(g),1)\cdot f)\in (\V_{\rho}\otimes F)_{\pi(g)\cdot x}.
\label{ecu:ac2}
\end{gather}
The above assignment defines an action of $\widetilde{G}_{\alpha}$
on $\V_{\rho}\otimes F$ and $\V_{\rho}\otimes F$
becomes a $\widetilde{G}_{\alpha}$-vector bundle.
Moreover for $(a,z)\in \widetilde{A}_{\alpha}$, as $M_{1}=1$, we have
\begin{gather*}
(a,z)\cdot (v\otimes
f)=M_{1}\tilde{\rho}\big(\tilde{\sigma}(1)^{-1}(a,z)\big)v\otimes
(1,1)\cdot f=\tilde{\rho}(a,z)v\otimes f
\end{gather*}
so that $\widetilde{A}_{\alpha}$ acts on $\V_{\rho}\otimes
F$ by the representation $\tilde{\rho}$; that is,
$p\colon \V_{\rho}\otimes F\rightarrow X$ is a
$\big(\widetilde{G}_{\alpha},\rho\big)$-equivariant vector bundle
over~$X$.
	
Finally, assume that $p\colon E \rightarrow X$ is a
$\big(\widetilde{G}_{\alpha},\rho\big)$-equivariant vector bundle over $X$.
By definition the map
\begin{align*}
\gamma\colon \ \V_{\rho}\otimes
\Hom_{\widetilde{A}_{\alpha}}(\V_{\rho},E)&\rightarrow E,
\\
(v,f)&\mapsto f(v)
\end{align*}
is an isomorphism of $\widetilde{A}_{\alpha}$-vector bundles. We may endow
$\V_{\rho}\otimes\Hom_{\widetilde{A}_{\alpha}}(\V_{\rho},E)$
with structure of a~$\widetilde{G}_{\alpha}$-vector bundle. That the map
$\gamma$ is an isomorphism of vector bundles and its
$\widetilde{G}_{\alpha}$-equivariance follows from next equations.
For $(g,z)\in \widetilde{G}_{\alpha}$ we have
\begin{align*}
\gamma((g,z)\cdot(v\otimes
f))&=\gamma\big(M_{\pi(g)}\tilde{\rho}\big(\tilde{\sigma}(\pi(g))^{-1}(g,z)\big)v\big)\otimes
(\pi(g),1)\cdot f\\
&=(\pi(g),1)\cdot f
\big(M_{\pi(g)}\tilde{\rho}\big(\tilde{\sigma}(\pi(g))^{-1}(g,z)\big)v\big)\\
&=\pi(g)\bullet\!
f\big(M_{\pi(g)}\tilde{\rho}\big(\tilde{\sigma}(\pi(g))^{-1}(g,z)\big)v\big)
\!=\tilde{\sigma}(\pi(g))\!\cdot
f\big(\tilde{\rho}\big(\tilde{\sigma}(\pi(g))^{-1}(g,z)\big)v\big)\\
	&=(g,z)\cdot f(v)=(g,z)\cdot\gamma(v\otimes f).
\end{align*}
The previous argument shows that $\gamma\colon
\V_{\rho}\otimes
\Hom_{\widetilde{A}_{\alpha}}(\V_{\rho},E)\rightarrow
E$ is an isomorphism of $\widetilde{G}_{\alpha}$-vector bundles.
	
Now, if $p\colon F\rightarrow X$ is a
$\widetilde{Q}_{\beta_{\rho,\alpha}}$-equivariant vector bundle over $X$
for which the central $\S^{1}$ acts by multiplication of
scalars, then by equation~(\ref{ecu:ac2}) we know that
$\V_{\rho}\otimes F$ is a
$\big(\widetilde{G}_{\alpha},\rho\big)$-equivariant vector bundle.
The canonical isomorphism of vector bundles
\begin{align*}
F&\rightarrow \Hom_{\widetilde{A}_{\alpha}}(
\V_{\rho}, \V_{\rho}\otimes F),
\\
x&\mapsto f_{x}\colon\ v\mapsto v\otimes x
\end{align*}
is in fact $\widetilde{Q}_{\beta_{\rho,\alpha}}$-equivariant, and therefore the vector
bundles $F$ and
$\Hom_{\widetilde{A}_{\alpha}}(\V_{\rho},\V_{\rho}\otimes
F)$ are isomorphic as
$\widetilde{Q}_{\beta_{\rho,\alpha}}$-equivariant vector bundles.
	
We conclude that the inverse map of the assignment
$[E]\mapsto[\Hom_{\widetilde{A}_{\alpha}}(\V_{\rho},E)]$
is precisely the map defined by the assignment $[F]\mapsto
[\V_{\rho}\otimes F]$.
\end{proof}


As an immediate corollary of Theorem~\ref{teo,bij} we obtain the following
identification of the $(\widetilde{G}_{\alpha}, \rho)$-equivariant $K$-theory groups of
Definition~\ref{def,grhoalfa} with the $\beta_{\rho,\alpha}$-twisted $Q$-equivariant
$K$-theory groups provided in Definition~\ref{usualdeftwisted}.

\begin{Cor}\label{identification}
Let $X$ be a $G$-space such that $A$ acts trivially on $X$. Assume that
$\rho$ is an~$\alpha$-re\-p\-resentation of $A$ such that
$g\cdot \rho\cong \rho$ for every $g\in G$. Then the assignment
\begin{align*}
K^{\ast}_{\widetilde{G}_{\alpha}, \rho}(X)
&\stackrel{\cong}{\rightarrow} \ ^{\beta_{\rho,\alpha}}K^{\ast}_{Q}(X),
\\
[E]&\mapsto [\Hom_{\widetilde{A}_{\alpha}}(\V_{\rho},E)]
\end{align*}
defines a natural isomorphism.
\end{Cor}

Suppose now that $\alpha$ is a normalized 2-cocycle on $G$ with values in $\S^{1}$.
Consider the action of $G$ on $\Irr^{\alpha}(A)$ constructed in Section~\ref{section cocycles}.
Given $[\tau]\in \Irr^{\alpha}(A)$ let $G_{[\tau]}=\{g\in G \mid g\cdot \tau\cong \tau\}$ denote
the isotropy subgroup of the action of $G$ at $[\tau]$. The group $G_{[\tau]}$ fits into
the short exact sequence
\begin{gather*}
1\rightarrow A\rightarrow G_{[\tau]}\stackrel{\pi}{\rightarrow} Q_{[\tau]}\rightarrow 1
\end{gather*}
and $Q_{[\tau]}=G_{[\tau]}/A$ agrees with the isotropy of the group $Q$ at $[\tau]$. Let
$\big(\widetilde{G}_{[\tau]}\big)_{\alpha}$ be the central extension corresponding to the cocycle
$\alpha$ seen as a cocycle defined on $G_{[\tau]}$. Therefore we have the following commutative
diagram of central extensions
\begin{equation*}
\begin{CD}
1 @>>> \S^{1}@>>>\big(\widetilde{G}_{[\tau]}\big)_{\alpha} @>>> G_{[\tau]} @>>> 1\\
@. @V{\text{id}}VV @VVV @VVV @. \\
1 @>>> \S^{1}@>>>\widetilde{G}_{\alpha} @>>> G @>>> 1.
\end{CD}
\end{equation*}
Assume that $X$ is a compact and Hausdorff $G$-space on which $A$ acts trivially.
As before we can extend the action of $G$ on $X$ to an action of $\widetilde{G}_{\alpha}$
on $X$ in such a way that $\S^{1}$ acts trivially. Let $p\colon E\to X$ be a
$\widetilde{G}_{\alpha}$-equivariant vector bundle on which $\S^{1}$ acts by scalar
multiplication on the fibers. As $\widetilde{A}_{\alpha}$ acts trivially on
$X$ each fiber of $E$ can be seen as an $\alpha$-representation of $A$. Using
fiberwise the canonical decomposition theorem for $\alpha$-representations
(Theorem~\ref{can,decom}) we see that the assignment
\begin{align*}
\gamma\colon\ \bigoplus_{[\tau]\in \Irr^{\alpha}(A)}\V_{\tau}\otimes
\Hom_{\widetilde{A}_{\alpha}}(\V_{\tau},E)&\rightarrow E,
\\
v\otimes f&\mapsto f(v)
\end{align*}
defines an isomorphism of $\widetilde{A}_{\alpha}$-equivariant vector bundles.
Using this decomposition we obtain the following theorem.

%
%
\begin{Th}\label{ThD}
Under the above assumptions there is a natural isomorphism
\begin{align*}
\Psi_{X}\colon\ \phantom{i}^{\alpha}K_{G}^{\ast}(X)&\rightarrow
\bigoplus_{[\tau]\in G\backslash \Irr^{\alpha}(A)}
K^{\ast}_{(\widetilde{G}_{[\tau]})_{\alpha},\tau}(X),
\\
[E]&\mapsto \bigoplus_{[\tau]\in G\backslash \Irr^{\alpha}(A)}
\big[\V_{\tau}\otimes\Hom_{\widetilde{A}_{\alpha}}(\V_{\tau},E)\big].
\end{align*}
This isomorphism is
functorial on maps $X\rightarrow Y$ of $G$-spaces on which $A$ acts trivially.
\end{Th}

\begin{proof}
The proof of this theorem follows the same lines of the proof of~\cite[Theorem 3.1]{GomezUribe}
so we only provide an outline of the proof.

Let us show first that the map $\Psi_{X}$ is well defined. To see this we have to
show that if $\rho$ is an $\alpha$-representation of $A$ then
$\V_{\rho}\otimes\Hom_{\widetilde{A}_{\alpha}}(\V_{\rho},E)$
has the structure of a $\big(\big(\widetilde{G}_{[\rho]}\big)_{\alpha},\rho\big)$-vector bundle.
Following the notation of Section~\ref{section cocycles} fix a set theoretical
section $\sigma\colon Q_{[\rho]}\to G_{[\rho]}$ for the projection map $\pi\colon G_{[\rho]}\to Q_{[\rho]}$
in such a way that $\sigma(1)=1$. Also, for every $q\in Q_{[\rho]}$ fix an~element
$M_{q}\in U(V_{\rho})$ such that
\begin{gather*}
\sigma(q)\cdot \rho(a) =\alpha\big(\sigma(q)^{-1}a,\sigma(q)\big)\alpha\big(\sigma(q),\sigma(q)^{-1}a\big)^{-1}
\rho\big(\sigma(q)^{-1}a\sigma(q)\big)=M_{q}^{-1}\rho(a)M_{q}.
\end{gather*}
This is possible as $\sigma(q)\in G_{[\rho]}$ so that we have $\sigma(q)\cdot \rho\cong \rho$.
We can choose $M_{1}=1$ as $\sigma(1)=1$.

Now if $(h,z)\in \big(\widetilde{G}_{[\rho]}\big)_{\alpha}$ and $v\otimes f\in
\V_{\rho}\otimes \Hom_{\widetilde{A}_{\alpha}}(\V_{\rho},E)$
we define $M_{(h,z)}\in U(V_{\rho})$ by
\begin{gather*}
M_{(h,z)}:=M_{\pi(h)}\tilde{\rho}\big(\tilde{\sigma}(\pi(h))^{-1}(h,z)\big),
\end{gather*}
where $\tilde{\rho}$ and $\tilde{\sigma}$ are defined in a similar way as in
Theorem~\ref{teo,bij}. Observe that $M_{(a,z)}=\tilde{\rho}(a,z)$ for all
$(a,z)\in \widetilde{A}_{\alpha}$. Moreover, given $(h,z)\in \big(\widetilde{G}_{[\rho]}\big)_{\alpha}$
and $v\otimes f\in \V_{\rho}\otimes \Hom_{\widetilde{A}_{\alpha}}(\V_{\rho},E)$
we define
\begin{gather*}
(h,z)\star(v\otimes f)=M_{(h,z)}v\otimes (h,z)\bullet f,
\end{gather*}
where $(h,z)\bullet f (w)=(h,z)f\big(M_{(h,z)}^{-1}w\big)$. Unraveling the definitions it
can be seen that this defines an action of
$(\widetilde{G}_{[\rho]})_{\alpha}$ on $\V_{\rho}\otimes \Hom_{\widetilde{A}_{\alpha}}(\V_{\rho},E)$
in such a way that the central factor $\S^{1}$ acts by multiplication of scalars.
This way $\V_{\rho}\otimes \Hom_{\widetilde{A}_{\alpha}}(\V_{\rho},E)$ has the structure
of a $\big(\widetilde{G}_{[\rho]}\big)_{\alpha}$-vector bundle and $A$ acts by the
$\alpha$-representation $\rho$ on the fibers so that
$\big[\V_{\rho}\otimes \Hom_{\widetilde{A}_{\alpha}}(\V_{\rho},E)\big]\in
K^{\ast}_{(\widetilde{G}_{[\rho]})_{\alpha},\rho}(X)$. This shows that
$\Psi_{X}$ is well defined.

Next we show that $\Psi_{X}$ is an isomorphism. For this write $\Irr^{\alpha}(A)=
\mathcal{A}_{1}\sqcup\mathcal{A}_{2}\sqcup\dots \sqcup\mathcal{A}_{k}$, where
$\mathcal{A}_{1},\mathcal{A}_{2},\dots, \mathcal{A}_{k}$ are the different
$G$-orbits of the action of $G$ on $\Irr^{\alpha}(A)$ defined in equation~(\ref{def:act}).
For every $1\leq i\leq k$ define
\begin{gather*}
E_{\mathcal{A}_{i}}=\bigoplus_{[\tau]\in \mathcal{A}_{i}}
\V_{\tau}\otimes \Hom_{\widetilde{A}_{\alpha}}(\V_{\tau},E).
\end{gather*}
Note that $\V_{\tau}\otimes \Hom_{\widetilde{A}_{\alpha}}(\V_{\tau},E)$ is an
$\widetilde{A}_{\alpha}$-equivariant vector bundle over $X$, so each
$E_{\mathcal{A}_{i}}$ is also an $\widetilde{A}_{\alpha}$-equivariant vector bundle over $X$
and the map
\begin{gather*}
\gamma\colon\ \bigoplus_{i=1}^{k}E_{\mathcal{A}_{i}}=\bigoplus_{[\tau]\in \Irr^{\alpha}(A)}
\V_{\tau}\otimes \Hom_{\widetilde{A}_{\alpha}}(\V_{\tau},E)\rightarrow E
\end{gather*}
defines an isomorphism of $\widetilde{A}_{\alpha}$-vector bundles. We are going to show that each
$E_{\mathcal{A}_{i}}$ is a $\widetilde{G}_{\alpha}$-vector bundle and that the map $\gamma$ is
$\widetilde{G}_{\alpha}$-equivariant. For this fix an index $1\leq i\leq k$ and an irreducible
$\alpha$-representation $\rho\colon A\rightarrow U(V_{\rho})$ such that
$[\rho]\in\mathcal{A}_{i}$. The elements in $\mathcal{A}_{i}$ can be written in the
form $[g_{1}\cdot\rho], \dots, [g_{r_{i}}\cdot\rho]$ for some elements
$g_{1}=1,g_{2},\dots,g_{r_{i}}\in G$. Therefore
\begin{gather*}
E_{\mathcal{A}_{i}}=\bigoplus_{j=1}^{r_{i}}\V_{g_{j}\cdot\rho}\otimes
\Hom_{\widetilde{A}_{\alpha}}(\V_{g_{j}\cdot\rho},E).
\end{gather*}
We can give a structure of
$\widetilde{G}_{\alpha}$-space on $E_{\mathcal{A}_{i}}$ in the following way.
Suppose that $(g,s)\in\widetilde{G}_{\alpha}$ and~that
$v\otimes f\in \V_{\rho}\otimes \Hom_{\widetilde{A}_{\alpha}}(\V_{\rho},E)_{x}$.
Decompose $gg_{j}$ in the form $gg_{j}=g_{l}h$, where $1\leq l\leq r_{i}$ and $h\in G_{[\rho]}$.
In other words $g_{l}$ is the representative chosen for the coset
$(gg_{j})G_{[\rho]}$ and $h=g_{l}^{-1}gg_{j}$. Then
\begin{gather*}
(g,s)(g_{j},1)=(g_{l},1)(h,z),
\end{gather*}
where $z=s\alpha(g,g_{j})\alpha(g_{l},h)^{-1}$. We define
\begin{gather*}
(g,s)\star (v\otimes f):=M_{(h,z)}v\otimes (g,s)\bullet
f \in \big(\V_{g_{l}\cdot\rho}\otimes
\Hom_{\widetilde{A}_{\alpha}}(\V_{g_{l}\cdot\rho},E)\big)_{(g,s)\cdot x},
\end{gather*}
where
\begin{gather*}
(g,s)\bullet f(w)=(g,s)f\big(M_{(h,z)}^{-1}w\big)
=(g,s)(h,z)^{-1}\tilde{\sigma}(\pi(h))f\big(M_{\pi(h)}^{-1}w\big).
\end{gather*}
It can be seen that this defines an action of $\widetilde{G}_{\alpha}\in E_{\mathcal{A}_{i}}$
for each $1\leq i\leq k$ making the vector bundle $E_{\mathcal{A}_{i}}$
into a $\widetilde{G}_{\alpha}$-vector bundle in such a way that the central factor
$\S^{1}$ acts by multiplication of scalars. Furthermore, the map
\begin{gather*}
\gamma\colon\ \bigoplus_{i=1}^{k}E_{\mathcal{A}_{i}} \rightarrow E
\end{gather*}
is an isomorphism of $\widetilde{A}_{\alpha}$-vector and the map $\gamma$ is
$\widetilde{G}_{\alpha}$-equivariant so that $\gamma$ is an isomorphism of~$\widetilde{G}_{\alpha}$-vector bundles. Now, the desired map $\Psi_{X}$ can be seen as the
direct sum $\oplus_{i=1}^{k}\Psi_{X}^{i}$ choosing for each $\mathcal{A}_{i}$ a representation
$[\tau_{i}]\in\mathcal{A}_{i}$, where each $\Psi_{X}^{i}$ is given by
\begin{align*}
\Psi_{X}^{i}\colon \phantom{i}^{\alpha}K_{G}^{\ast}(X)&\rightarrow
K^{*}_{\left(\widetilde{G}_{[\tau_{i}]}\right)_{\alpha},\tau_{i}}(X),
\\
[E]&\mapsto \big[\V_{\tau_{i}}\otimes\Hom_{\widetilde{A}_{\alpha}}(\V_{\tau_{i}},E)\big].
\end{align*}
In what follows we will construct the map
$\zeta_{i}\colon K^{*}_{\left(\widetilde{G}_{[\tau_{i}]}\right)_{\alpha},\tau_{i}}(X)
\rightarrow \phantom{i}^{\alpha}K_{G}^{\ast}(X)$
which will be the right inverse of $\Psi_{X}^{i}$. Take $\rho=\tau_{i}$ and consider
a vector bundle $F\in\text{Vect}_{\left(\widetilde{G}_{[\rho]}\right)_{\alpha},\rho}(X)$.
Fix $g_{1}=1,g_{2},\dots,g_{r}$ representatives for the different cosets in $G/G_{[\rho]}$.
Let
\begin{gather*}
L_{F}:=\bigoplus_{j=1}^{r}\big[(g_{j},1)^{-1}\big]^{\ast} F.
\end{gather*}
Using ideas similar to the ones used in~\cite[Theorem 3.1]{GomezUribe} we can endow $L_{F}$
with the structure of a $\widetilde{G}_{\alpha}$-vector bundle on which the
central factor $\S^{1}$ acts by multiplication of scalars in such a~way that the action
of $\big(\widetilde{G}_{[\rho]}\big)_{\alpha}$ on $\big[(g_{1},1)^{-1}\big]^{\ast} F\cong F$ agrees with
the given action of $\big(\widetilde{G}_{[\rho]}\big)_{\alpha}$ on~$F$. We~define
\begin{align*}
\zeta_{i}\colon\ K^{*}_{\left(\widetilde{G}_{[\tau_{i}]}\right)_{\alpha},\tau_{i}}(X)
&\rightarrow \phantom{i}^{\alpha}K_{G}^{\ast}(X),
\\
[F]&\mapsto [L_{F}].
\end{align*}
Now, since we have the isomorphism
$\V_{\rho}\otimes \Hom_{\widetilde{A}_{\alpha}}
\big(\V_{\rho},\oplus_{j=1}^{r}\big[(g_{j},1)^{-1}\big]^{\ast} F\big)\cong F$ as
$\big(\widetilde{G}_[\rho]\big)_{\alpha}$ vector bundles. We obtain at the level of $K$-theory that
$\zeta_{i}$ is a right inverse for $\Psi_{i}$ so that
$\zeta=\oplus_{i=1}^{r}\zeta_{i}$ is a right inverse for $\Psi_{X}$. In a similar way, using the
work given above it can be seen that $\zeta$ is also a left inverse so that the
map $\Psi_{X}$ is indeed an isomorphism.

To finish, we observe that functoriality follows from the fact that if $\tau$ is an
$\alpha$-representation of $A$ then the bundles
$\V_{\tau}\otimes \Hom_{\widetilde{A}_{\alpha}}(\V_{\tau},f^{\ast} E)$ and
$f^{\ast}\big(\V_{\tau}\otimes \Hom_{\widetilde{A}_{\alpha}}(\V_{\tau},E)\big)$
are canonically isomorphic as $((\widetilde{G}_{[\tau]})_{\alpha},\tau)$-equivariant bundles
whenever $f\colon Y\rightarrow X$ is a
$G$-equivariant map from spaces on which $A$ acts trivially.
\end{proof}


As a result of Theorem~\ref{ThD} and Corollary~\ref{identification} we obtain the following
theorem that is the main result of this article.

\begin{Th}\label{Thmain}
Suppose that $A$ is a normal subgroup of a finite group $G$. Let $\alpha$
be a normalized 2-cocycle on $G$ with values in $\S^{1}$ and $X$ a
compact $G$-space on which $A$ acts trivially.
Then there is a natural isomorphism
\begin{align*}
\Phi_{X}\colon\ \phantom{i}^{\alpha}K_{G}^{\ast}(X)&\rightarrow
\bigoplus_{[\tau]\in G\backslash \Irr^{\alpha}(A)}\phantom{i}
^{\beta_{\tau,\alpha}}K^{\ast}_{Q_{[\tau]}}(X),
\\
[E]&\mapsto \bigoplus_{[\tau]\in G\backslash \Irr^{\alpha}(A)}
\big[\Hom_{\widetilde{A}_{\alpha}}(\V_{\tau},E)\big].
\end{align*}
Here $\beta_{\tau,\alpha}$ is the $2$-cocycle associated to $\tau$ and $\alpha$ as defined in equation~$(\ref{defQbeta})$. This isomorphism is functorial on maps $X\rightarrow Y$ of
$G$-spaces on which $A$ acts trivially.
\end{Th}

%
%

\section{Atiyah--Hirzebruch spectral sequence}\label{Section 4}

As an application of Theorem~\ref{Thmain} we obtain a formula for
the third differential in the Atiyah--Hirzebruch spectral sequence for $\alpha$-twisted $G$-equivariant $K$-theory
under suitable hypotheses. The treatment in this section generalizes the one given in~\cite[Section 5]{GomezUribe}.

To start, assume that $A$ is a normal subgroup of a finite group $G$ and let $Q=G/A$.
Suppose that $Q$ acts freely on a compact and Hausdorff space $X$. Let $\beta\colon Q\times Q\to \S^{1}$
be a normalized $2$-cocycle with values in $\S^{1}$. As is explained in~\cite[equation~3.6]{GomezUribe},
under these hypotheses the $Q$-equivariant twisted $K$-group $\phantom{i}^{\beta}K_{Q}^{\ast}(X)$ can
be described as a non-equivariant twisted $K$-group over the space $X/Q$. To formulate this,
let
\begin{gather*}
1\to \S^{1}\to \widetilde{Q}_\beta\to Q\to 1
\end{gather*}
be the central extension corresponding to $\beta$. Fix a separable Hilbert space~$\calh$
endowed with a unitary linear action of $\widetilde{Q}_{\beta}$ such that the central factor~$\S^{1}$
acts by multiplication of scalars and such that all the irreducible
representations of this kind appear infinitely number of times in~$\calh$. \mbox{Under}~these hypotheses, the
space of Fredholm operators $\text{Fred}(\calh)$ classifies $\beta$-twisted $Q$-equivariant $K$-theory. Thus
there is a natural isomorphism
$\phantom{i}^{\beta}K_{Q}^{0}(X)\cong [X,\text{Fred}(\calh)]_{Q}$. In~addition, observe that
the projective unitary group $PU(\calh)$ is an Eilenberg--Maclane space of type~$K(\Z,2)$.
On the other hand, as $\S^{1}$ acts by multiplication of scalars on $\calh$, the action of
$\widetilde{Q}_\beta$ on~$\calh$ induces
a commutative diagram of central extensions of the form
\begin{gather*}
\begin{CD}
1 @>>> \S^{1}@>>>\widetilde{Q}_{\beta}@>>> Q @>>> 1\\
@. @V{\text{id}}VV @V{\tilde{\phi}_{\beta}}VV @VV{\phi_{\beta}}V @. \\
1 @>>> \S^{1}@>>>U(\calh) @>>> PU(\calh) @>>> 1.
\end{CD}
\end{gather*}
Upon passage to classifying spaces, we obtain the map $B\phi_{\beta}\colon BQ\to BPU(\calh)$
and $BPU(\calh)$ is an Eilenberg--Maclane space of type $K(\Z,3)$ so that the homotopy class of
$B\phi_{\beta}$ corresponds to a cohomology class $\big[\beta^\Z\big]\in H^{3}(BQ;\Z)$. We remark
that the class $\big[\beta^\Z\big]$ agrees with the class $[\beta]\in H^{2}(Q;\S^{1})$ defined by the
cocycle $\beta$ under the standard identification $H^{2}\big(Q;\S^{1}\big)\cong H^{3}(BQ;\Z)$. In~addition,
as $X$ is a free $Q$-space there is a unique up to homotopy $Q$-equivariant map $X \to EQ$ inducing
a map $h\colon X/Q \to BQ$ at the level of the quotient spaces. This way we obtain the continuous
map $f_{\beta}:=B\phi_\beta\circ h\colon X/Q\to BPU(\calh)$ which is precisely
the data needed to define the non-equivariant $f_{\beta}$-twisted $K$-groups. We remark
that the cohomology class associated to~the homotopy class of the map $f_{\beta}$ is precisely
$h^{*}\big(\big[\beta^\Z\big]\big)\in H^{3}(X/Q;\Z)$. By~\cite[equation~(3.6)]{GomezUribe} we~have an isomorphism
\begin{gather}\label{eqvsnoneq}
\phantom{i}^{\beta}K_{Q}^{\ast}(X)\cong K^{*}(X/Q;f_{\beta}).
\end{gather}
We refer the reader to~\cite[Section 3]{GomezUribe} for the details on this construction.

Assume now that $X$ is a compact $G$-CW complex $X$ in such a way that $G_{x}=A$ for every $x\in A$.
Thus we are assuming that the $G$ action on $X$ has constant isotropy subgroups.
This is equivalent to asking that the group $Q$ acts freely on $X$. Fix $\alpha\in Z^2\big(G,\S^{1}\big)$
a normalized $2$-cocycle. Let $\calu=\{U_{i}\}_{i\in \cali}$ (where $\cali$ is a well
ordered set) be a contractible slice cover of~$X$ by~$G$-invariant open sets.
Thus for every sequence $i_{1}\le \cdots\le i_{p}$ of elements in $\cali$ with
$U_{i_{1},\dots,i_{p}}:=U_{i_{1}}\cap\cdots\cap U_{i_{p}}$ nonempty
we can find some element $x_{i_{1},\dots,i_{p}}\in U_{i_{1},\dots, i_{p}}$
such that the inclusion map
$Gx_{i_{1},\dots,i_{p}}\hookrightarrow U_{i_{1},\dots,i_{p}}$ is a
$G$-homotopy equivalence. It~can be seen that such a~contractible slice cover exists
for any compact $G$-CW complex. Using the $G$-cover $\calu$ we can construct a spectral sequence
akin to the Atiyah--Hirzebruch spectral sequence that converges to $\phantom{i}^{\alpha}K_{G}^{\ast}(X)$.
This spectral sequence can be constructed in the same way as in the case of equivariant $K$-theory
explained in~\cite[Section 5]{Segal-Classifying}. The $E_{2}$-page of this
spectral sequence is such that $E_2^{p,q}=0$ if $q$ is odd. For even values of $q$ we have that
$E_2^{p,q}=H^p_G(X, R^{\alpha}(-))$, the $p$-th Bredon cohomology group $H^p_G(X, R^{\alpha}(-))$
with coefficients in $R^{\alpha}(-)$. Here $R^{\alpha}(-)$ denotes the coefficient system given by the
$\alpha$-twisted representation groups. Explicitly, this coefficient
systems assigns $R^{\alpha}(H)$ to the coset $G/H$. From here it follows automatically that the
differential ${\rm d}_{2}$ is trivial and thus $E^{*,*}_{2}=E^{*,*}_{3}$.

\begin{Th}
Under the above hypotheses, the $E_{3}$-term in the Atiyah--Hirzebruch spectral sequence is such that
\begin{gather*}
E_{3}^{p,q} \cong
\begin{cases}
\bigoplus_{[\tau]\in Q \backslash \Irr^{\alpha}(A)}H^{p}\big(X/Q_{[\tau]};\Z\big) &\text{if}\quad q
\ \text{is even},
\\
0& \text{if}\quad q\ \text{is odd}.
\end{cases}
\end{gather*}
Furthermore, the differential
\begin{gather*}
{\rm d}_{3}\colon\ \bigoplus_{[\tau]\in Q \backslash \Irr^{\alpha}(A)}H^{p}\big(X/Q_{[\tau]};\Z\big) \to
\bigoplus_{[\tau]\in Q \backslash \Irr^{\alpha}(A)}H^{p}\big(X/Q_{[\tau]};\Z\big)
\end{gather*}
is defined coordinate-wise in such a way that for $\eta\in H^{p}\big(X/Q_{[\tau]};\Z\big)$ we have
\begin{gather*}
{\rm d}_3(\eta) = Sq^3_\Z \eta - h^*\big[\beta_{\tau,\alpha}^{\Z}\big] \cup \eta.
\end{gather*}
Here $\beta_{\tau,\alpha}$ is the $2$-cocycle associated to $\tau$ and $\alpha$ as defined in equation~$(\ref{defQbeta})$. Also, $Sq^3_\Z$ is the composition of the maps $\beta \circ Sq^2 \circ \text{mod}_2$, where $\text{mod}_2$
is the reduction modulo $2$, $Sq^2$ is the Steenrod operation, and $\beta$ is the Bockstein map for the coefficient sequence
$0\to\Z \xrightarrow{2}\Z \to \Z/2\to 0$.
\end{Th}

\begin{proof}
By assumption the action of $G$ on $X$ has constant isotropy, therefore the Bredon coho\-mo\-logy groups
$H^*_G(X, R^{\alpha}(-))$ can be identified with the cohomology
of the cochain comp\-lex $\Hom_{\Z[G]}(C_{*}(X),R^{\alpha}(A))$. The group $A$ acts trivially on both
$C_{*}(X)$ and $R^{\alpha}(A)$ so that this cochain comp\-lex is isomorphic to
$\Hom_{\Z[Q]}(C_{*}(X),R^{\alpha}(A))$. As a $Q$-module $R^{\alpha}(A)$ is~a~per\-mutation module.
Therefore, as a $Q$-representation we have
an isomorphism
\begin{gather*}
R^{\alpha}(A)\cong \bigoplus_{[\tau]\in Q \backslash\Irr^{\alpha}(A)} \Z\big[Q/Q_{[\tau]}\big].
\end{gather*}
Via this isomorphism we can identify $H_G^p(X,R^{\alpha}(-))$ with
$\bigoplus_{[\tau]\in Q \backslash\Irr^{\alpha}(A)}H^{p}\big(X/Q_{[\tau]};\Z\big)$. For even values of $q$
it follows that $E_{3}^{p,q}=E_{2}^{p,q}=\bigoplus_{[\tau]\in Q \backslash \Irr^{\alpha}(A)}H^{p}(X/Q_{[\tau]};\Z)$.
This proves the first part of the theorem since we already know that for odd values of $q$ we have
$E_{3}^{p,q}=E_{2}^{p,q}=0$.

To determine the third differential we use Theorem~\ref{Thmain} to obtain an isomorphism
\begin{gather*}
\Phi_{X}\colon \ \phantom{i}^{\alpha}K_{G}^{\ast}(X)\rightarrow
\bigoplus_{[\tau]\in G\backslash \Irr^{\alpha}(A)}\phantom{i}
^{\beta_{\tau,\alpha}}K^{\ast}_{Q_{[\tau]}}(X).
\end{gather*}
By hypothesis, the group $Q$ acts freely on $X$ and thus for each $[\tau]\in G\backslash \Irr^{\alpha}(A)$
the group $Q_{[\tau]}$ also acts freely on $X$. This together with (\ref{eqvsnoneq})
provides an isomorphism
\begin{gather*}
\phantom{i}^{\alpha}K_{G}^{\ast}(X)\cong
\bigoplus_{[\tau]\in Q \backslash \Irr^{\alpha}(A)} K^{*}\big(X/Q_{[\tau]}; f_{\beta_{\tau,\alpha}}\big).
\end{gather*}
For each $[\tau]\in Q \backslash \Irr^{\alpha}(A)$ passing to the quotient we obtain an open cover
$\calu/Q_{[\tau]}:=\!\big\{U_{i}/Q_{[\tau]}\big\}_{i\in \cali}$ of the quotient space $X/Q_{[\tau]}$. Using the
cover $\calu/Q_{[\tau]}$ we obtain a spectral sequence that converges to
$K^{*}\big(X/Q_{[\tau]}; f_{\beta_{\tau,\alpha}}\big)$. The naturality of Theorem~\ref{Thmain} implies that
the spectral sequence computing $\phantom{i}^{\alpha}K_{G}^{\ast}(X)$ decomposes as a direct sum
of the spectral sequences associated to the non-equivariant twisted K-theories
$K^{*}\big(X/Q_{[\tau]}; f_{\beta_{\tau,\alpha}}\big)$. As it was pointed out above,
the cohomology class associated to the homotopy class of the map $f_{\beta_{\tau,\alpha}}$ is precisely
$h^{*}\big(\big[\beta_{\tau,\alpha}^\Z\big]\big)\in H^{3}(X/Q;\Z)$. \mbox{Using}~\cite[Proposition~4.6]{Atiyah-Segal2} we conclude
that the third differential in the Atiyah--Hirzebruch spectral sequence to
$K^{*}\big(X/Q_{[\rho]}; f_{\beta_{\tau,\alpha}}\big)$ is the operator
\begin{align*}
{\rm d}^\rho_3\colon\ H^{*}(X/Q_{[\tau]};\Z) &\to H^{*+3}\big(X/Q_{[\tau]};\Z\big),
\\
\eta&\mapsto d^\tau_3(\eta)= Sq^3_\Z \eta - h^*\big[\beta_{\tau,\alpha}^\Z\big] \cup \eta.
\end{align*}
This proves the theorem.
\end{proof}

\section{Examples}\label{Section 5}

In this section we explore some examples of Theorems~\ref{ThD} and~\ref{Thmain}
for the dihedral groups $D_{2n}$, where $n\ge 2$ an even integer.


We start by considering first the particular case where $G=D_{8}$. The group $D_{8}$
is generated by the elements $a,b$ subject to the relations $a^{4}=b^{2}=1$ and $bab=a^{3}$.
Let $\alpha\colon D_{8}\times D_{8}\rightarrow \S^{1}$ be the 2-cocycle defined by
\begin{gather*}
\alpha\big(a^{l},a^{j}b^{k}\big)=1\qquad \text{and}\qquad
\alpha\big(a^{l}b,a^{j}b^{k}\big)=i^{j}
\quad \text{for} \quad 0\leq j,l\leq 3 \quad \text{and}\quad k=0,1.
\end{gather*}
Note that $\alpha$ is a nontrivial normalized 2-cocycle such that its corresponding
cohomology class defines the generator of $H^{2}\big(D_{8};\S^{1}\big)\cong \Z/2$.
By Example~\ref{ej,D2n}, taking $n=4$, we know that up to isomorphism $D_{8}$ has two
irreducible projective $\alpha$-representations
$\tau_{1}$ and $\tau_{2}$ defined by
\begin{gather*}
\tau_{l}\big(a^{j}b^{k}\big)=A_{l}^{j}B_{l}^{k} \qquad \text{for} \quad 0\leq j,k\leq 3
\quad \text{and} \quad l=0,1.
\end{gather*}
In the above definition we have
\begin{gather*}
A_{1}=
\begin{pmatrix}
i & 0\\
0 & 1
\end{pmatrix},\qquad
A_{2}=
\begin{pmatrix}
-1 & 0\\
0 & -i
\end{pmatrix}
\qquad \text{and}\qquad
B_{1}=B_{2}=
\begin{pmatrix}
0 & 1\\
1 & 0
\end{pmatrix}\!.
\end{gather*}
With this in mind we are going to explore the following examples of Theorems
\ref{ThD} and~\ref{Thmain}.

\begin{Ex}\label{Example1} Suppose first that $G=D_{8}$ and $A=\Z/4=\langle a\rangle$. Therefore
\begin{gather*}
Q=G/A=\{[1],[b]\}\cong \Z/2.
\end{gather*}
Let us take $X$ to be the space with only one point $\ast$ equipped with the trivial
$D_{8}$-action. In~this case $^{\alpha}K^{\ast}_{D_{8}}(\ast)= R^{\alpha}(D_{8})$, where
$R^{\alpha}(D_{8})$ denotes the $\alpha$-twisted representation group of $D_{8}$.
As~pointed out above $\tau_{1}$ and $\tau_{2}$ are the only irreducible
$\alpha$-representations of $D_{8}$ and thus we have an isomorphism of abelian groups
\begin{gather*}
\phantom{i}^{\alpha}K^{\ast}_{D_{8}}(\ast)\cong R^{\alpha}(D_{8})
=\Z\tau_{1}\oplus\Z\tau_{2}.
\end{gather*}
On the other hand, observe that
$\alpha|_{A}$ is trivial so that
\begin{gather*}
\Irr^{\alpha}(A)=\Irr(A)=\big\{[1],[\rho],\big[\rho^{2}\big], \big[\rho^{3}\big]\big\},
\end{gather*}
where $\rho\colon A\rightarrow \C$ is the irreducible representation defined by $\rho(a)=i$.
For the action of $D_{8}$ on~$\Irr^{\alpha}(A)$ we have
\begin{gather*}
b\cdot\rho(a)=\alpha(ba,b)\alpha(b,ba)^{-1}\rho(bab)
=\alpha\big(a^{3}b,b\big)\alpha\big(b,a^{3}b\big)^{-1}\rho(a^{3})
=\big(i^{3}\big)^{-1}i^{3}=1
\end{gather*}
therefore $b\cdot \rho=1$. Moreover,
\begin{gather*}
b\cdot\rho^{2}(a)=\alpha(ba,b)\alpha(b,ba)^{-1}\rho^{2}(bab)
=\alpha\big(a^{3}b,b\big)\alpha\big(b,a^{3}b\big)^{-1}\rho^{2}\big(a^{3}\big)
=\big(i^{3}\big)^{-1}\big(\rho^{2}(a)\big)^{3}=-i
\end{gather*}
so that $b\cdot\rho^{2}=\rho^{3}$. We conclude that orbits of the $D_{8}$ action on
$\Irr^{\alpha}(A)$ are $\{[1],[\rho]\}$ and $\big\{\big[\rho^{2}\big],\big[\rho^{3}\big]\big\}$. Thus we can choose
$[1]$ and $\big[\rho^{2}\big]$ as representatives for the elements in $D_{8}\backslash \Irr^{\alpha}(A)$
and $G_{[1]}=G_{[\rho^{2}]}=A$.
In this case Theorem~\ref{ThD} gives us an isomorphism
\begin{gather*}
\Psi\colon\ R^{\alpha}(D_{8})=^{\alpha}K^{\ast}_{D_{8}}(\ast)\stackrel{\cong}{\rightarrow}
K^{\ast}_{(\widetilde{G}_{[1]})_{\alpha},1}(\ast)\oplus
K^{\ast}_{(\widetilde{G}_{[\rho^{2}]})_{\alpha},\rho^{2}}(\ast).
\end{gather*}
As $G_{[1]}=G_{[\rho^{2}]}=A$ and $\alpha$ restricted to $A$ is trivial we have that
$\big(\widetilde{G}_{[1]}\big)_{\alpha}=\big(\widetilde{G}_{[\rho^{2}]}\big)_{\alpha}=A\times \S^{1}$.
Therefore $K^{\ast}_{(\widetilde{G}_{[1]})_{\alpha},1}(\ast)\cong \Z\tilde{1}$ and
$K^{\ast}_{(\widetilde{G}_{[\rho^{2}]})_{\alpha},\rho^{2}}(\ast)\cong \Z\tilde{\rho^{2}}$.
(Recall that $\tilde{\rho^{2}}$ denotes the representation of
$\big(\widetilde{G}_{[\rho^{2}]}\big)_{\alpha}$ on which $\S^{1}$ acts by multiplication of
scalars corresponding to $\rho^{2}$ and similarly for $\tilde{1}$).
For the representations $\tau_{1}$ and $\tau_{2}$ we have
 \begin{gather*}
 \arraycolsep=-.05ex
\tau_{1}(a)=
\left(\,\begin{matrix}
 \fbox{$i$} &0 \\[-.2ex]
 0 &\fbox{1}
 \end{matrix}\,\right) \qquad\text{and}\qquad
\tau_{2}(a)=
\left(\,\begin{matrix}
 \fbox{$-1$} &0 \\[-.2ex]
 0 & \fbox{$-i$}
 \end{matrix}\,\right)\!.
 \end{gather*}
Thus as $A$-representations $\tau_{1}$ is isomorphic to $1\oplus \rho$ and
$\tau_{2}$ is isomorphic to $\rho^{2}\oplus\rho^{3}$.
Moreover, in the isomorphism given by Theorem~\ref{ThD} we have
\begin{align*}
\Psi\colon\ R^{\alpha}(D_{8})\cong \Z\tau_{1}\oplus\Z\tau_{2}
&\stackrel{\cong}{\rightarrow} K^{\ast}_{(\widetilde{G}_{[1]})_{\alpha},1}(\ast)\oplus
K^{\ast}_{(\widetilde{G}_{[\rho^{2}]})_{\alpha},\rho^{2}}\cong
\Z{\tilde{1}}\oplus \Z\tilde{\rho^{2}},
\\
\tau_{1}&\mapsto \tilde{1},
\\
\tau_{2}&\mapsto \tilde{\rho^{2}}.
\end{align*}
On the other hand, by Theorem~\ref{Thmain} we have an isomorphism
\begin{gather*}
\Phi\colon\ R^{\alpha}(D_{8})\stackrel{\cong }{\rightarrow} \phantom{i}
^{\beta_{1,\alpha}}K^{\ast}_{Q_{[1]}}(\ast)\oplus \phantom{i}
^{\beta_{\rho^{2},\alpha}}K^{\ast}_{Q_{[\rho^{2}]}}(\ast).
\end{gather*}
As $G_{[1]}=G_{[\rho^{2}]}=A$ we have that $Q_{[1]}=Q_{[\rho^{2}]}=\{1\}$ is the
trivial group. Therefore $\beta_{1,\alpha}$ and~$\beta_{\rho^{2},\alpha}$ are
the trivial cocycles and Theorem~\ref{Thmain} gives us the isomorphism
\begin{gather*}
\Phi\colon\ R^{\alpha}(D_{8})\cong \Z\tau_{1}\oplus\Z\tau_{2} \stackrel{\cong}{\rightarrow}
K^{\ast}_{\{1\}}(\ast)\oplus K^{\ast}_{\{1\}}(\ast)\cong \Z\oplus \Z.
\end{gather*}
\end{Ex}

\begin{Ex} Suppose now that $G=D_{8}$ and $A=Z(D_{8})=\big\langle a^{2}\big\rangle\cong \Z/2$.
Therefore in this case we have
\begin{gather*}
Q=G/A=\{[1],[b],[a],[ab]\}\cong\Z/2\oplus \Z/2.
\end{gather*}
As in the previous example $\alpha|_{A}$ is trivial and
\begin{gather*}
\Irr^{\alpha}(A)=\Irr(A)=\{[1],[\sigma]\},
\end{gather*}
where $\sigma\colon A\rightarrow \C$ is the representation defined by $\sigma\big(a^{2}\big)=-1$.
For the action of $D_{8}$ on $\Irr^{\alpha}(A)$ we have
\begin{gather*}
b\cdot\sigma\big(a^{2}\big)=\alpha\big(ba^{2},b\big)\alpha\big(b,ba^{2}\big)^{-1}\sigma\big(ba^{2}b\big)
=\alpha\big(a^{2}b,b\big)\alpha\big(b,a^{2}b\big)^{-1}\sigma\big(a^{2}\big)
=(-1)(-1)=1
\end{gather*}
therefore $b\cdot \sigma=1$. In particular the action of
$D_{8}$ on $\Irr^{\alpha}(A)$ is transitive and we can choose $[1]$
as a representative for the set $D_{8}\backslash \Irr^{\alpha}(A)$. For the
representation $1$ we have $G_{[1]}=\langle a\rangle\cong \Z/4$.
If we take again $X=\ast$ then Theorem~\ref{ThD}
gives us an isomorphism
\begin{gather*}
\Psi\colon\ R^{\alpha}(D_{8})\cong \Z\tau_{1}\oplus\Z\tau_{2}
\stackrel{\cong}{\rightarrow}
K^{\ast}_{(\widetilde{G}_{[1]})_{\alpha},1}(\ast).
\end{gather*}
As $\alpha$ is trivial on $\langle a\rangle$ we have that
$\big(\widetilde{G}_{[1]}\big)_{\alpha}=\langle a\rangle\times \S^{1}$. If
$\rho\colon \langle a\rangle\to \C$ denotes the representation defined by $\rho(a)=i$
then $\tilde{1}$ and $\tilde{\rho^{2}}$ can be seen as
$\big(\big(\widetilde{G}_{[1]}\big)_{\alpha},1\big)$-vector bundles over $\ast$ and
\begin{gather*}
K^{\ast}_{(\widetilde{G}_{[1]})_{\alpha},1}(\ast)\cong \Z\tilde{1}\oplus \Z\tilde{\rho^{2}}.
\end{gather*}
Observe that
\begin{gather*}
 \arraycolsep=-.05ex
\tau_{1}(a^{2})=
\left(\,\begin{matrix}
 \fbox{$-1$} &0 \\[-.2ex]
 0 &\fbox{1} \end{matrix}\,\right) \qquad\text{and}\qquad
\tau_{2}(a^{2})=
\left(\,\begin{matrix}
 \fbox{1} &0 \\[-.2ex]
 0 & \fbox{$-1$}
 \end{matrix}\,\right)\!.
 \end{gather*}
Therefore as $A$-representations we have $\tau_{1}\cong\tau_{2}\cong 1\oplus \sigma$.
However, as $\langle a\rangle$-representations $\tau_{1}$ is isomorphic to
$1\oplus \rho$ and $\tau_{2}$ is isomorphic to $\rho^{2}\oplus\rho^{3}$. It follows that
in the isomorphism given by Theorem~\ref{ThD} we have
\begin{align*}
\Psi\colon\ R^{\alpha}(D_{8})\cong \Z\tau_{1}\oplus\Z\tau_{2}
&\stackrel{\cong}{\rightarrow}
K^{\ast}_{(\widetilde{G}_{[1]})_{\alpha},1}(\ast)\cong \Z\tilde{1}\oplus \Z\tilde{\rho^{2}},\\
\tau_{1}&\mapsto \tilde{1},\\
\tau_{2}&\mapsto \tilde{\rho^{2}}.
\end{align*}
On the other hand, by Theorem~\ref{Thmain} we have an isomorphism
\begin{gather*}
\Phi\colon\ R^{\alpha}(D_{8})\stackrel{\cong }{\rightarrow} \phantom{i}
^{\beta_{1,\alpha}}K^{\ast}_{Q_{[1]}}(\ast).
\end{gather*}
In this case $Q_{[1]}=\langle a\rangle/\big\langle a^{2}\big\rangle\cong \Z/2$.
The cocycle $\beta_{1,\alpha}$ is the trivial cocycle and
Theorem~\ref{Thmain} gives us an isomorphism
\begin{gather*}
\Phi\colon\ R^{\alpha}(D_{8})\cong \Z\tau_{1}\oplus\Z\tau_{2} \stackrel{\cong}{\rightarrow}
K^{\ast}_{\Z/2}(\ast)=R(\Z/2).
\end{gather*}
For the group $\Z/2$ we have $R(\Z/2)\cong \Z1\oplus \Z s$, where $s$ denotes the
sign representation. The isomorphism $\Phi$ maps $\tau_{1}$ to $1$ and $\tau_{2}$ to
$s$.
\end{Ex}

\begin{Ex} Example~\ref{Example1} can easily be generalized to the dihedral groups $D_{2n}$
with $n$ an~even number. Suppose then that $n$ is an even number and let $D_{2n}$
be the group generated by~the elements $a,b$ subject to the relations $a^{n}=b^{2}=1$
and $bab=a^{-1}$. Fix $\epsilon$ a primitive $n$-th root of unity and let
\begin{gather*}
\alpha\colon\ D_{2n}\times D_{2n}\rightarrow \S^{1}
\end{gather*}
be the function defined by
\begin{gather*}
\alpha\big(a^{j},a^{k}b^{l}\big)=1 \qquad \text{and} \qquad \alpha\big(a^{j}b,a^{k}b^{l}\big)
=\epsilon^{k}\qquad \text{for}\quad 0\leq j,k\leq n-1 \quad \text{and}\quad l=0,1.
\end{gather*}
The function $\alpha$ defines a normalized $2$-cocycle on $D_{2n}$ with values in $\S^{1}$ whose
corresponding cohomology class is the generator in $H^{2}\big(D_{2n},\S^{1}\big)\cong \Z/2$.
By Example~\ref{ej,D2n} we know that the irreducible projective $\alpha$-representations of
$D_{2n}$ are $\tau_{i}\colon D_{2n}\rightarrow {\rm GL}_{2}(\C)$ for $i=1,\dots, n/2$, defined~by
\begin{gather*}
\tau_{i}\big(a^{k}b^{l}\big)=A_{i}^{k}B_{i}^{l} \qquad\text{for} \quad 0\leq k\leq n-1 \quad \text{and}\quad l=0,1,
\end{gather*}
where
\begin{gather*}
A_{i}=
\begin{pmatrix}
\epsilon^{i} & 0\\
0 & \epsilon^{1-i}
\end{pmatrix}
\quad \text{and}\quad
B_{i}=
\begin{pmatrix}
0 & 1\\
1 & 0
\end{pmatrix}\!.
\end{gather*}
If we take $X=\ast$ endowed with the trivial $D_{2n}$ action we have
\begin{gather*}
\phantom{i}^{\alpha}K^{\ast}_{D_{2n}}(\ast)\cong R^{\alpha}(D_{2n})
=\Z\tau_{1}\oplus\cdots \oplus \Z\tau_{n/2}.
\end{gather*}
Take $A=\langle a\rangle \cong\Z/n$ so that $Q=D_{2n}/A=\{[1],[b]\}\cong \Z/2$.
Observe that $\alpha|_{A}$ is trivial and thus
\begin{gather*}
\Irr^{\alpha}(A)=\Irr(A)=\big\{[1],[\rho],\big[\rho^{2}\big],\dots, \big[\rho^{n-1}\big]\big\},
\end{gather*}
where $\rho(a)=\epsilon$. For the action of $D_{2n}$ on $\Irr^{\alpha}(A)$ we have that
\begin{gather*}
b\cdot 1=\rho,\qquad b\cdot \rho^{2}=\rho^{n-1},\qquad
\dots,\qquad b\cdot \rho^{n/2}=\rho^{n/2+1}
\end{gather*}
so that orbits of the $D_{2n}$ action on $\Irr^{\alpha}(A)$ are
$\big\{[1], [\rho]\big\}, \big\{\big[\rho^{2}\big], \big[\rho^{n-1}\big]\big\}, \dots, \big\{\big[\rho^{n/2}\big], \big[\rho^{n/2+1}\big]\big\}$.
We~can choose $[\rho], \big[\rho^{2}\big],\dots, \big[\rho^{n/2}\big]$ as representatives for the set
$D_{2n}\backslash \Irr^{\alpha}(A)$ and we have $G_{[\rho]}=G_{[\rho^{2}]}=\cdots =
G_{[\rho^{n/2}]}=A$. As $\alpha|_{A}$ is trivial we have
$\big(\widetilde{G}_{[\rho^{i}]}\big)_{\alpha}=A\times \S^{1}$ for $i=1, \dots, n/2$.
In this case Theorem~\ref{ThD} gives us the isomorphism
\begin{align*}
\Psi\colon\ R^{\alpha}(D_{2n})
=\Z\tau_{1}\oplus\cdots \oplus \Z\tau_{n/2}
&\stackrel{\cong}{\rightarrow}
\bigoplus_{i=1}^{n/2}K^{\ast}_{(\widetilde{G}_{[\rho^{i}]})_{\alpha},\rho^{i}}(\ast)
\cong \bigoplus_{i=1}^{n/2}
 \Z\tilde{\rho^{i}},
 \\
\tau_{i}&\mapsto \tilde{\rho}^{i}.
\end{align*}
On the other hand, for $i=1, \dots, n/2$ we have $Q_{[\rho^{i}]}=\{1\}$ and
$\beta_{\rho^{i},\alpha}$ is the trivial cocycle. In this case
Theorem~\ref{Thmain} gives us an isomorphism
\begin{gather*}
\Phi\colon\ R^{\alpha}(D_{2n})\cong \Z\tau_{1}\oplus\cdots \oplus \Z\tau_{n/2}
\stackrel{\cong}{\rightarrow} \bigoplus_{i=1}^{n/2}K^{\ast}_{\{1\}}(\ast)\cong \Z^{n/2}.
\end{gather*}
\end{Ex}

\subsection*{Acknowledgements}

The first author acknowledges and thanks the financial support provided by
MINCIENCIAS through grant number FP44842-013-2018 of the
Fondo Nacional de Financiamiento para la~Ciencia, la Tecnolog\'ia y la Innovaci\'on.
The second author acknowledges and thanks the financial support provided by MINCIENCIAS
through grant number 727 of the program Doctorados nacio\-nales 2015 of the Fondo Nacional de
Financiamiento para la Ciencia, la Tecnolog\'ia y~la~Inno\-vaci\'on. Additionally, the authors would
like to thank the referees and the editor for providing useful comments that helped improve this manuscript.

\pdfbookmark[1]{References}{ref}
\LastPageEnding

\end{document}